\RequirePackage{fix-cm}
\documentclass{svjour3}                     
\usepackage{graphicx}
\usepackage{amsmath}
\usepackage{amssymb}
\usepackage{times, color}
\usepackage{graphicx,epsfig}
\usepackage{exscale, fancybox,cite}

\def\smartqed{\hfill $\triangle$}
\def\inte{\mathop{\rm Int }}
\def\dom{\mathop{\rm Dom\,}}

\renewcommand{\thetheorem}{\arabic{section}.\arabic{theorem}}

\newcommand{\R}{\mathbb R}
\newcommand{\N}{\mathbb N}

\newcommand{\Multi}{\,\,\lower 1pt
\hbox{$\overrightarrow{\longrightarrow}$}\,\,}

\begin{document}

\title{Local nonsmooth Lyapunov pairs for first-order evolution differential inclusions}


\author{Samir Adly \and
    Abderrahim Hantoute \and Michel Th\'era 
}


\institute{S. Adly\at
Universit\'e de Limoges, Laboratoire XLIM.
123, avenue Albert Thomas, 87060 Limoges CEDEX,
          France\\
              \email{samir.adly@unilim.fr}           
           \and
           A. Hantoute \at
             University of Chile, CMM Piso 7, Avenida Blanco Encalada 2120, Santiago, Chile .\at
                \email{ahantoute@dim.uchile.cl}
                \and
                M. Th\'era \at
                Universit\'e de Limoges, Laboratoire XLIM.
123, avenue Albert Thomas, 87060 Limoges CEDEX,
          France\\
                \email{michel.thera@unilim.fr}
}

\date{Received: date / Accepted: date}

\maketitle

\begin{abstract}
 The general theory of Lyapunov's stability of first-order differential inclusions in Hilbert spaces has been studied by the authors  in a previous work \cite{AHT2012}. This new contribution  focuses on the natural case when the maximally monotone operator governing the given  inclusion has a domain with nonempty interior.
 This setting permits to have nonincreasing Lyapunov functions on the whole trajectory of the solution to the given differential inclusion.  It also allows some more explicit  criteria for Lyapunov's pairs. Some consequences to the viability of closed sets  are  given, as well as some useful cases relying on the continuity  or/and convexity of the involved functions. Our analysis makes use of standard tools from convex and variational analysis.
\keywords{Evolution differential inclusions, lower semi-continuous
 functions, invariant sets,  proximal subdifferential, Clarke  subdifferential, Fr\'echet  subdifferential, limiting proximal subdifferential, maximally monotone operator, strong solution, weak solution,  Lyapunov pair.
}
 \subclass{37B25, 47J35, 93B05}
\end{abstract}

\section{Introduction and notations} In various  applications modeled by ODE's, one may be forced to work with systems that have non-differentiable solutions. Also,    
 Lyapunov's functions,   that is    positive definite functions whose decay along the trajectories of the system, 
which are  used to establish a stability property of the system,  may  be nondifferentiable.  The need to extend the classical differentiable Lyapunov's stability to the nonsmooth case is unavoidable when studying stability properties of discontinuous systems. In practice, many systems in physics, engineering, biology etc exhibit generally nonsmooth energy functions, which   are usually a typical candidates for Lyapunov functions; thus elements of nonsmooth analysis become essential \cite{ClarketalBook,SP94,BC99,KPS2006}. 
A typical example is given by   the case  of  piecewise linear dynamical systems  called\textit{ Linear Complementarity Systems}  (LCS) for which the analysis of asymptotic and exponential stability  uses a  piecewise quadratic Lyapunov function \cite{KPS2006}.
Let us remind that LCS are defined as follows:
 $$LCS(A,B,C,D)\quad \quad \left\{\begin{array}{l}
  {\displaystyle\dot{x}(t;x_{0})=Ax(t)+Bu(t) },\;\;x(t_0)=x_0,\\
0\leq u(t)\perp Cx+Du\geq 0,
\end{array}
\right.
$$
where $A\in\mathbb{R}^{n\times n}$, $B\in\mathbb{R}^{n\times m}$, $C\in\mathbb{R}^{m\times n}$ and $D\in\mathbb{R}^{m\times m}$ are real matrices, $x_0$ is the initial condition, $\dot{x}$ is the time derivative of the trajectory $x(t)$ and $a\perp  b$ means that the two vectors $a$  and $b$ are orthogonal.
Linear and nonlinear complementarity problems belong to the more general mathematical formalism of  \textit{Differential Variational Inequalities} (DVI),  introduced by J.S. Pang and D. Stewart \cite{PS08}. It is a combination of an ordinary differential equation (ODE) with a variational inequality or a complementarity constraint. A DVI consists to find trajectories $t\mapsto x(t)$ and $t\mapsto u(t)$ such that
$$DVI(f,F,K)\quad \quad \left\{
\begin{array}{l}
  {\color{black}\dot{x}(t)=f(t,x(t),u(t)) },\;\;x(t_0)=x_0,\\
{\color{black}\langle F(t,x(t),u(t)),v-u(t) \rangle} \geq 0,\;\forall v\in K, u(t)\in K\mbox{ for a.e. } t\geq t_0,
\end{array}
\right.
$$
where $K$ is a closed convex subset of   a Hilbert space $H$, $f$ and $F$ are {\color{black}given mappings. }
When $K$ is a closed convex cone, then problem $DVI(f,F,K)$ is equivalent to a  Differential Complementarity Problem (DCP):
$$DCP(f,F,K)\quad \quad \left\{
\begin{array}{l}
  {\color{black}\dot{x}(t)=f(t,x(t),u(t)) },\;\;x(t_0)=x_0,\\
K\ni u(t)\perp F(t,x(t),u(t))\in K^*,\mbox{ a.e. } t\geq t_0.
\end{array}
\right.
$$
Since  DVI and DCP formalisms  unify  several known  mathematical problems such that ordinary differential {\color{black}equations with } discontinuous right-hand term, differential
algebraic  equations, dynamic complementarity problems etc .. (see \cite{cps1,cps2} for more details), 
it was proved to be powerful for the treatment of many problems in science and engineering such that: unilateral contact problems in mechanics, finance, traffic networks, electrical circuits etc \ldots. 
According  also to the fact that LCS formalism has   many  of applications in various  areas  including for instance robotics, economics, finance, non smooth  mechanics, etc (see Camlibel, Pang and  Shen,  \cite{KPS2006} and the monograph by Facchinei and   Pang, \cite{FP2003}), it has received recently a great interest from the  mathematical programming and control communities from the theoretical and numerical point of view. 
\vskip 2mm
Instead of considering LCSs or DVIs,   throughout  this contribution we are interested in the general framework of infinite-dimensional dynamical systems, that is systems  of the form:
\begin{equation}
\dot{x}(t;x_{0})\in f(x(\cdot; x_{0}))-Ax(\cdot; x_{0}),\text{ \ }x_{0}%
\in\operatorname*{cl}(\dom  A)\text{ \ }a.e.\text{ \ }t\geq0.
\label{id}%
\end{equation}
Here, and thereafter, $\operatorname*{cl}(\dom A)$ is the 
closure of the domain of a maximally monotone operator $A:H\rightrightarrows H$
defined on a   real  Hilbert space $H$,   possibly nonlinear and multivalued  with domain $\dom A$ and $f$ is a Lipschitz  continuous mapping defined on
$\operatorname*{cl}(\dom  A).$



A pair of proper  lower semicontinuous (lsc for short)  functions $V,W:H\rightarrow\mathbb{R}\cup\{+\infty\}$ is
said to form a Lyapunov pair for (\ref{id}) if for all $x_{0}\in
\operatorname*{cl} (\dom  A)$ the solution of (\ref{id}), in a
sense that will be precised in Section \ref{santiago}, denoted by
$x(\cdot; \cdot,x_{0})$ satisfies
\begin{equation}
V(x(t; x_{0}))-V(x(s; x_{0}))+\int_{s}^{t}W(x(\tau;x_{0}))d\tau\leq0 \text{\ for
all }t\geq s\geq0. \label{vw}%
\end{equation}
Observe that when $W\equiv0$ one recovers the classical notion of Lyapunov
functions; e.g., \cite{smirnov02}. More generally, instead of (\ref{vw}), we
are going to consider functions $V,W$ satisfying for some $a\geq0$
\[
e^{at}V(x(t; x_{0}))-e^{as}V(x(s;x_{0}))+\int_{s}^{t}W(x(\tau;x_{0}%
))d\tau\text{\ for all }t\geq s\geq0.
\]
In this case, the (weighted) pair $(V,W)$ will be refered to as a $a$-Lyapunov
pair. The main motivation in using $a$-Lyapunov pairs instead of simply
functions is that many stability concepts for the equilibrium sets of
(\ref{id}), namely stability, asymptotic or finite-time stability, can be
obtained just by choosing appropriate functions $W$ in (\ref{vw}). The weight
$e^{at}$ is useful for instance when exponential stability is concerned. So,
even in autonomous systems like those of (\ref{id}), the function $W$ or the
weight $e^{at}$ may be of     a certain  utility since, in some sense, it emphasizes the
decreasing of the Lyapunov function $V$.

The method of Lyapunov functions  is a corner stone of  the study
of the controllability and stabilizability of control systems. Its  history is rich and has  been described in several places  and various seminal contributions has been made  to the subject.   We refer to  Clarke \cite{Clarke2009,Clarke} for  an  overview 
 of the recent developments of  the theory where he pointed out that   for nonlinear systems, Lyapunov's method  turns out to be essential to consider nonsmooth Lyapunov functions, even if the
underlying control dynamics are themselves smooth.

  Over the years, among the  various  contributions, Kocan \& Soravia  \cite{Kocanetall02},  characterized Lyapunov's pairs in terms of viscosity solutions of a related  partial differential inequality.

 Another well-established approach consists of characterizing Lyapunov's pairs by  means of the contingent derivative of the maximally monotone operator $A$, see for instance C\^{a}rj\u{a} \& Motreanu  \cite{carja-motreanu1}, for the case of a maximally linear monotone operator and also  when $A$ is a multivalued m-accretive operator on an arbitrary  Banach space  \cite{Carjaetal09-flow}.  In these approaches the authors  used   tangency and flow-invariance arguments combined with a priori estimates and approximation.

The  starting point of this contribution is  the paper  by Adly \& Goeleven
  \cite{Adlyetal04} 
  in which smooth Lyapunov functions were
used
 in the framework of the
  second order differential equations,
and non-linear mechanical systems with frictional unilateral constraints.

In   this article  we  provide  a different   approach that don't make use of   viscosity solutions or contingent derivatives associated to the operator $A$.  Our objective is to emphasize   our previous contribution \cite{AHT2012} to the setting where the involved maximally monotone operator has a domain with nonempty interior. This case  includes the finite dimensional framework  since in this case  the relative interior of the domain of the operator is always nonempty. Moreover, the criteria for Lyapunov's pairs are checked only in the interior of the domain (or the relative interior) instead of the closure of the whole  domain  as in \cite{Adlyetal04}.  In contrary to \cite{Adlyetal04}, this setting also ensures obtaining global Lyapunov's pairs   and permits in this way to control the whole trajectory of the solution to the given differential inclusion.

The summary of the paper is as follows. In Section 2 we introduce the main
tools and basic results used in the paper. In Section 3 we give a new primal
and dual criteria for lower semicontinuous Lyapunov pairs. This is achieved in Proposition
\ref{heartlocal} and  Theorem \ref{heartlocalfinite}. In Section 4, we make a review of
some old and recent criteria for Lyapunov pairs. Section 5 is dedicated to
complete the proofs of the main results given in Section 3.

\section{Notation and main tools}

\   Throughout the paper, $H$ is  a (real) Hilbert space endowed with the inner (or scalar)
product $\langle\cdot,\cdot\rangle$ and the associated norm  is denoted by $\left\Vert
\cdot\right\Vert $.   We identify $H^{\ast}$ (the space of continuous linear
functionals defined on $H$) to $H,$ and  we denote the weak limits ($w-\lim,$ for
short) by the symbol $\rightharpoonup$ to distinguish it from the usual symbol
$\rightarrow$ used for strong limits. The zero vector in $H$ is denoted by
$\theta.$

 We start this section  by reviewing   some notations   used  throughout the paper. Given a
nonempty set $S\subset H$ (or $S\subset H\times\mathbb{R}$), by
$\operatorname*{co}S$, $\operatorname*{cone}S$, and $\operatorname*{aff}S$, we
denote the \emph{convex hull}, the \emph{conic hull}, and the \emph{affine
hull} of the set $S$, respectively. Moreover, $\inte S$ is the
topological \emph{interior} of $S$, and $\operatorname*{cl}S$ and
$\overline{S}$ are indistinctly used for the \emph{closure} of $S$ (with
respect to the norm topology on $H)$.
 We also use $\operatorname*{cl}^{w}S$
or $\overline{S}^{w}$ when we deal with the closure of $S$ with respect to the
weak topology. We note $\operatorname*{ri}S$ the (topological) \emph{relative
interior} of $S$, i.e., the interior of $S$ in the topology relative to
$\operatorname*{aff}S$ whatever this set is nonempty (see\cite[Chapter 6] {roc70} for more on this fundamental
notion).
For $x\in H$ (or $x\in
H\times\mathbb{R}$), $\rho\geq0,$ $B_{\rho}(x)$ is the open ball with center
$x$ and radius $\rho,$ and $\overline{B}_{\rho}(x)$ is the closure of
$B_{\rho}(x)$, while $B:=B_{1}(\theta)$ stands for the unit open ball. For
$a,b\in\overline{\mathbb{R}} :=\mathbb{R\cup\{+\infty}, \mathbb{-\infty\}}$ we
denote $[a,b)$ the interval closed at $a$ and open at $b\ $($[a,b],(a,b),...$
are defined similarly); hence $\mathbb{R}_{+}:=[0,\infty)$. Finally, for
$\alpha\in\mathbb{R},$ we note $\alpha^{+}$ for $\max\{0,\alpha\}$.

Our  notation   is the standard one used    in convex and variational analysis and  in monotone
operator theory; see, e.g., \cite{BrezisBook,RockBook2002}. The
\emph{indicator function} of $S$ is the function defined as
\[
\mathrm{I}_{S}(x):=\left\{
\begin{array}
[c]{ll}%
0 & \text{if }x\in S\\
+\infty & \text{otherwise.}%
\end{array}
\right.
\]
The \emph{distance function} to $S$ is denoted by
\[
d(x,S):=\inf\{\left\Vert x-y\right\Vert \mid y\in S\},
\]
and the \emph{orthogonal projection} on $S$, $\pi_{S}$, is defined as
\[
\pi_{S}(x):=\{y\in S\mid\left\Vert x-y\right\Vert =d(x,S)\}.
\]
If $S$ is closed and convex, $S_{\infty}\subset H$ (or $H\times\mathbb{R}%
$)\ denotes its \emph{recession cone}:
\[
S_{\infty}:=\{y\mid x+\lambda y\in S\text{ for some }x\text{ and all }
\lambda\geq0\},
\]
while, $S^{\circ}\subset H$ (or $H\times\mathbb{R}$) denotes the polar of $S$
given by
\[
S^{\circ}:=\{y\mid\langle y,v\rangle\leq1\text{ for all }v\in S\}.
\]
Given a function $\varphi:H\rightarrow\overline{\mathbb{R}}$,
its\emph{\ (effective) domain} and \emph{epigraph} are defined by
\[
\dom \varphi:=\{x\in H\mid\varphi(x)<+\infty\},
\]
\[
\operatorname*{epi}\varphi:=\{(x,\alpha)\in H\times\mathbb{R}\mid
\varphi(x)\leq\alpha\}.
\]
For $\lambda\in\mathbb{R}$, the \emph{open sublevel set} of $\varphi$ at
$\lambda$ is
\[
\lbrack\varphi>\lambda]:=\{x\in H\mid\varphi(x)>\lambda\};
\]
$[\varphi\geq\lambda],$ $[\varphi\leq\lambda]$, and $[\varphi<\lambda]$ are
defined similarly. We say that $\varphi$ is proper if $\dom %
\varphi\neq\emptyset$ and $\varphi(x)>-\infty$ for all $x\in H.$ We say that
$\varphi$ is convex if $\operatorname*{epi}\varphi$ is convex, and (weakly)
lower semicontinuous (lsc, for short) if $\operatorname*{epi}\varphi$ is
closed with respect to the (weak topology) norm-topology on $H$. We denote
\[
\mathcal{F}(H):=\{\varphi:H\rightarrow\overline{\mathbb{R}}\mid\varphi\text{
is proper and lsc}\},
\]%
\[
\mathcal{F}_{w}(H):=\{\varphi:H\rightarrow\overline{\mathbb{R}}\mid
\varphi\text{ is proper and weakly lsc}\};
\]
$\mathcal{F}(H;{\mathbb{R}}_{+})$ and $\mathcal{F}_{w}(H;{\mathbb{R}}_{+})$
stand for the subsets of nonnegative functions of $\mathcal{F}(H)$ and
$\mathcal{F}_{w}(H)$, respectively.

  As maximally monotone set-valued  operators play an important role in this work, it is useful to recall some of basic definitions and  some of their properties. More generally, they  have frequently shown themselves to be a key class of objects in both modern Optimization and Analysis; see, e.g.,    \cite{BC2011,BrezisBook,RockBook2002,simons,Borwein10,BorYao2012}.
\vskip 2mm\noindent
For an operator $A:H\rightrightarrows H,$ the \emph{domain} and the
$\emph{graph}$ of $A$ are given respectively by
\[
\dom A:=\{z\in H\mid Az\neq\emptyset\}\text{ and }%
\operatorname*{gph}A:=\{(x,y)\in H\times H\mid y\in Ax\};
\]
for notational simplicity we identify the operator $A$ to its graph. The
$\emph{inverse}$ $\emph{operator}$ of $A$, denoted by $A^{-1}$, is defined as
\[
(y,x)\in A^{-1}\Longleftrightarrow(x,y)\in A.
\]
We say that an  operator $A$ is   \emph{monotone} if
\[
\langle y_{1}-y_{2},x_{1}-x_{2}\rangle\geq0\text{ for all}
(x_{1},y_{1}),(x_{2},y_{2})\in A,
\]
and \emph{maximally monotone} if $A$  is monotone and has no proper monotone extension (in the sense of graph inclusion).  If $A$ is maximally monotone, it is well known (e.g.,
\cite{simons}) that $\overline{\dom A}$ is convex, and $Ax$ is
convex and closed for every $x\in\dom A$. Moreover, if
$\inte (\dom A)\neq\emptyset$, then
$\inte (\dom A)$ is convex, $\inte %
(\dom A)=\inte (\overline{\dom A})$,
and $A$ is bounded locally on $\inte (\dom A)$. Note that the 
domain or the range of a maximally monotone
operator may fail to be convex, see, e.g.,\cite[page 555]{RockBook2002}.
In particular, if $A$ is the subdifferential $\partial\varphi$ of some lower semicontinuous (lsc for short)
convex and  proper function $\varphi:H\rightarrow\overline{\mathbb{R}}$, then $A $
is a classical  example of a maximally monotone operator,  as is a  linear operator with a positive symmetric part.
We know that
\[
\dom A\subset\dom \varphi\subset\overline
{\dom  \varphi}=\overline{\dom A}.
\]
For $x\in\dom A,$ we shall use the notation $(Ax)^{\circ}$ to
denote the \emph{principal} \emph{section} of $A,$ i.e., the set of points of
minimal norm in $Ax$.

Nonsmooth and variational analysis play a central role in this study. Hence,
we need to recall briefly some concepts used through the paper. More details
can be found for instance in
\cite{BorweinZhuBook,ClarketalBook,ClarkeBook,MordokBook,RockBook2002}. We assume that
$\varphi\in\mathcal{F}(H),$ and take $x\in\dom \varphi.$

A vector $\xi\in H$ is called a \emph{proximal subgradient} of $\varphi$ at
$x$, written $\xi\in\partial_{P}\varphi(x),$ if there are $\rho>0$ and
$\sigma\geq0$ such that
\[
\varphi(y)\geq\varphi(x)+\langle\xi,y-x\rangle-\sigma\left\Vert y-x\right\Vert
^{2}\text{ \ for all }y\in B_{\rho}(x);
\]
the \emph{domain of }$\partial_{P}\varphi$ is then given by
\[
\dom \partial_{P}\varphi:=\{x\in H\mid\partial_{P}\varphi
(x)\neq\emptyset\}.
\]
The set $\partial_{P}\varphi(x)$ is convex, possibly empty and not necessarily closed.

A vector $\xi\in H$ is called a \emph{Fr\'echet subgradient} of $\varphi$ at
$x$, written $\xi\in\partial_{F}\varphi(x),$ if
\[
\varphi(y)\geq\varphi(x)+\langle\xi,y-x\rangle+o(\left\Vert y-x\right\Vert ).
\]
Associated to  proximal and  Fr\'echet subdifferentials, limiting
objects have been introduced. A vector $\xi\in H$ belongs to the
\emph{limiting proximal subdifferential} of $\varphi$ at $x$, written
$\partial_{L}\varphi(x)$, if there exist sequences $(x_{k})_{k\in{\mathbb{N}}%
}$ and $(\xi_{k})_{k\in{\mathbb{N}}}$ such that $x_{k}\underset {\varphi}{\rightharpoonup} 
x$ (that is, $x_{k}\rightharpoonup x$ and $\varphi(x_{k}%
)\rightarrow\varphi(x)$), $\xi_{k}\in\partial_{P}\varphi(x_{k})$ and $\xi
_{k}\rightharpoonup\xi.$

A vector $\xi\in H$ is called a \emph{horisontal subgradient} of $\varphi$ at
$x$, written $\xi\in\partial_{\infty}\varphi(x),$ if there exist sequences
$(\alpha_{k})_{k\in{\mathbb{N}}}\subset{\mathbb{R}}_{+}, (x_{k} )_{k\in
{\mathbb{N}}}$ and $(\xi_{k} )_{k\in{\mathbb{N}}}$ such that $\alpha
_{k}\rightarrow0^{+},$ $x_{k} \underset {\varphi}{\rightarrow} x$, $\xi_{k}\in
\partial_{P}\varphi(x_{k})$ and $\alpha_{k}\xi_{k}\rightharpoonup\xi.$

The \emph{Clarke subdifferential} of $\varphi$ at $x$ is defined by the
following so-called \emph{representation formula}; see, e.g., Mordukhovich
\cite{MordokBook} and Rockafellar \cite{RockBook2002},
\[
\partial_{C}\varphi(x)=\overline{\operatorname*{co}}^{w}\{\partial_{L}%
\varphi(x)+\partial_{\infty}\varphi(x)\}.
\]
From a geometrical point of view, if $S\subset H$ is closed and $x\in S,$ the
\emph{proximal normal cone to} $S$ \emph{at} $x$ is
\[
\mathrm{N}_{S}^{P}(x):=\partial_{P}\mathrm{I}_{S}(x).
\]
We also denote  by $\mathrm{\widetilde{N}} _{S}^{P}(x)$ the subset of
$\mathrm{N}_{S}^{P}(x)$ given by
\[
\mathrm{\widetilde{N}} _{S}^{P}(x):=\{\xi\in H\mid\langle\xi,y-x\rangle
\leq\left\Vert y-x\right\Vert ^{2}\text{ \ for all }y\in S\text{ closed to
}x\}.
\]
It can be proved; e.g., \cite{ClarkeBook}, that
\[
\mathrm{N}_{S}^{P}(x)=\left\{
\begin{array}
[c]{l}%
\operatorname*{cone}(\pi_{S}^{-1}(x)-x),\text{ \ \ if }\pi_{S}^{-1}%
(x)\neq\emptyset,\\
\{\theta\}\text{ \ \ if }\pi_{S}^{-1}(x)=\emptyset,
\end{array}
\right.
\]
where $\pi_{S}^{-1}(x):=\{y\in H\setminus S\mid x\in\pi_{S}(y)\}.$

Similarly, $\mathrm{N}_{S}^{L}(x):=\partial_{L}\mathrm{I}_{S}(x)$
($=\partial_{\infty}\mathrm{I}_{S}(x)$) is the \emph{limiting normal cone to
}$S$ \emph{at }$x,$ and $\mathrm{N}_{S}^{C}(x):=\overline{\operatorname*{co}%
}^{w}\{\mathrm{N}_{S}^{L}(x)\}$ is the \emph{Clarke normal cone to }%
$S$\emph{\ at }$x.$

In that way, the above subdifferentials of $\varphi\in\mathcal{F}(H)$\ can be
geometrically described as
\begin{align*}
\partial_{P}\varphi(x)  &  =\{\xi\in H\mid(\xi,-1)\in\mathrm{N}%
_{\operatorname*{epi}\varphi}^{P}(x,\varphi(x))\},\\
\partial_{\infty}\varphi(x)  &  =\{\xi\in H\mid(\xi,0)\in\mathrm{N}%
_{\operatorname*{epi}\varphi}^{P}(x,\varphi(x))\}.
\end{align*}

We call \emph{contingent cone} \emph{to} $S$\emph{\ at }$x\in S$ (or the
\emph{Bouligand tangent cone}), written $\mathrm{T}_{S}(x),$ the cone given
by
\[
\mathrm{T}_{S}(x):=\{\xi\in H\mid x+\tau_{k}\xi_{k}\in S\text{ for some }%
\xi_{k}\rightarrow\xi\text{ and } \tau_{k}\rightarrow0^{+}\}.
\]
The Dini directional derivative of the function $\varphi$ $(\in\mathcal{F}%
(H))$ at $x\in\dom \varphi$ in the direction $v\in H$ is given
by
\[
\varphi^{\prime}(x,v)=\liminf_{t\rightarrow0^{+},w\rightarrow v}\frac
{\varphi(x+tw)-\varphi(x)}{t}.
\]
Hence, $\operatorname*{epi}\varphi^{\prime}(x,\cdot)=\mathrm{T}%
_{\operatorname*{epi}\varphi} (x,\varphi(x)).$ The \emph{G\^{a}teaux
derivative} of $\varphi$ at $x$ is a linear continuous form on $H,$ written
$\varphi_{G}^{\prime}(x),$ satisfying
\[
\lim_{t\rightarrow0^{+}}\frac{\varphi(x+tv)-\varphi(x)}{t}=\langle\varphi
_{G}^{\prime}(x),v\rangle\text{ \ \ for all }v\in H.
\]
We close this section by giving some properties of the subdifferential sets
defined above that  will be used later on. First, it follows easily from the
definitions that
\[
\partial_{P}\varphi(x)\subset\partial_{F}\varphi(x)\subset\partial_{L}%
\varphi(x)\subset\partial_{C}\varphi(x).
\]
If $\varphi$ is convex, then
\[
\partial_{P}\varphi(x)=\partial_{C}\varphi(x)=\partial\varphi(x),
\]
where $\partial\varphi(x)$ is the usual  \emph{Moreau-Rockafellar}
subdifferential of $\varphi$ at $x:$
\[
\partial\varphi(x):=\{\xi\in H\mid\varphi(y)-\varphi(x)\geq\langle
\xi,y-x\rangle\text{ for all }y\in H\}.
\]
If $\varphi\in\mathcal{F}(H)$ is G\^{a}teaux-differentiable at $x\in
\dom \varphi,$ we have
\[
\partial_{P}\varphi(x)\subset\{\varphi_{G}^{\prime}(x)\}\subset\partial
_{C}\varphi(x).
\]
If $\varphi$ is $C^{1}$ then
\[
\partial_{P}\varphi(x)\subset\{\varphi^{\prime}(x)\}=\partial_{C}%
\varphi(x)\text{ and } \partial_{\infty}\varphi(x)=\{\theta\}.
\]
If $\varphi$ is $C^{2}$ then
\[
\partial_{P}\varphi(x)=\partial_{C}\varphi(x)=\{\varphi^{\prime}(x)\}.
\]
In particular, if $\varphi:=d(\cdot,S)$ with $S\subset H$ closed, for $x\in S
$ we have that
\[
\partial_{C}\varphi(x)=\mathrm{N}_{S}^{C}(x)\cap B,
\]
while, for $x\not \in S$ such that $\partial_{P}\varphi(x)\neq\emptyset,$
$\pi_{S}(x)$ is a singleton and (e.g., \cite{ClarketalBook})
\[
\partial_{P}\varphi(x)=\frac{x-\pi_{S}(x)}{\varphi(x)};
\]
hence
\[
\partial_{L}\varphi(x)=\left\{  w-\lim_{k}\frac{x_{k}-\pi_{S}(x_{k})}%
{\varphi(x)};\, x_{k}\rightharpoonup x\right\}  .
\]
More generally, we have that
\[
\mathrm{N}_{S}^{P}(x)=\mathbb{R}_{+}\partial_{P}d_{S}(x)\text{ and }%
\mathrm{N}_{S}^{C}(x) =\overline{\mathbb{R}_{+}\partial_{C}d_{S}(x)}^{w}%
\]
(with the convention that $0.\emptyset=\{\theta\}$).

Finally, we recall that $\varphi\in\mathcal{F}(\mathbb{R})$ is nonincreasing
if and only if $\xi\leq0$ for every $\xi\in\partial_{P}\varphi(x)$ and
$x\in\mathbb{R},$   (e.g.,   \cite{ClarketalBook}). We shall use the following version of the Gronwall Lemma (e.g.,
\cite[Lemma 1]{Adlyetal04}).

\begin{lemma}
\label{monoton}Given $t_{2}>t_{1}\geq0$, $a\neq0$, and $b\geq0,$ we assume
that an  absolutely continuous function $\psi:[t_{1},t_{2}]\rightarrow
\mathbb{R}_{+}$ satisfies
\[
\psi^{\prime}(t)\leq a\psi(t)+b\text{ \ \ a.e. }t\in\lbrack t_{1},t_{2}].
\]
Then, for all $t\in\lbrack t_{1},t_{2}],$
\[
\psi(t)\leq(\psi(t_{1})+\frac{b}{a})e^{a(t-t_{1})}-\frac{b}{a}.
\]

\end{lemma}

\section{Local characterization of Lyapunov pairs on the interior of the
domain of $A$}\label{santiago}

In this section we provide the desired explicit criterion for lower
semicontinuous (weighted-)\ Lyapunov pairs associated to the differential
inclusion (\ref{id}):%
\[
\dot{x}(t; x_{0})\in f(x(\cdot; x_{0}))-Ax(\cdot;x_{0}),\text{ }x_{0}%
\in\operatorname*{cl}\left(  \dom A\right)  ,
\]
where $A:H\rightrightarrows H$ is a maximally monotone operator and
$f:\operatorname*{cl}\left(  \dom A\right)  \subset H\rightarrow
H$ is a Lipschitz continuous mapping. Recall that for fixed $T>0$ and
$x_{0}\in\operatorname*{cl}\left(  \dom A\right)  ,$ a  
\emph{ strong solution } of (\ref{id}), $x(\cdot; x_{0}):[0,T]\rightarrow H,$ is a
uniquely defined absolute continuous function which satisfies $x(0; x_{0}%
)=x_{0}$\ together with (see, e.g., \cite{BrezisBook})%

\begin{align}
&  \dot{x}(t; x_{0})\in L_{\operatorname*{loc}}^{\infty}%
((0,T],H),\label{liploc}\\
& \nonumber\\
&  x(t; x_{0})\in\dom A,\,\,\,\text{for all }%
t>0,\label{apartenance}\\
& \nonumber\\
&  \dot{x}(t; x_{0})\in f(x(t; x_{0}))-Ax(t; x_{0}),\quad\text{a.e. }t\geq0.
\label{aevery}%
\end{align}
Existence of strong solutions is known to occur if for instance:
\begin{itemize}
\item $x_{0}%
\in\dom A$, $\inte\left(  \operatorname*{co}%
\left(  \dom A\right)  \right)  \neq\emptyset;$ 
\item $\dim H<\infty;$
\item or if $A\equiv\partial\varphi$ where  $\varphi:H\rightarrow\mathbb{R\cup
\{+\infty\}}$ is  a lsc extended-real-valued convex proper function.
\end{itemize}
Moreover, we have that $\dot{x}(\cdot;x_{0})\in L^{\infty}([0,T],H)$ if and
only if $x_{0}\in\dom A$. In this later case, $x(\cdot;x_{0})$
is derivable from right at each $s\in\left[  0,T\right)  $ and
\[
\frac{d^{+}x(\cdot;x_{0})}{t}(s)=f(x(s;x_{0}))-\pi_{Ax(s;x_{0})}%
(f(x(s;x_{0}))).
\]
The strong solution also satisfies the so-called semi-group property,
\begin{equation}
x(s;x(t; x_{0}))=x(s+t; x_{0})\text{ for all }s,t\geq0, \label{sgp}%
\end{equation}
together with the relationship
\[
\left\Vert x(t; x_{0})-x(t; y_{0})\right\Vert \leq e^{L_{f}t}\left\Vert
x_{0}-y_{0}\right\Vert
\]
whenever $t\geq0$ and $x_{0},y_{0}\in\operatorname*{cl}(\dom %
A);$ hereafter, $L_{f}$ denotes the Lipschitz constant of the mapping $f$ on
$\operatorname*{cl}(\dom A).$

In the general case, it is well established that (\ref{id}) admits a unique
\emph{weak} \emph{solution} $x(\cdot; x_{0})\in C(0,T;H)$ which satisfies
$x(t; x_{0})\in\operatorname*{cl}(\dom A)$ for all $t\geq0$.
More precisely, there exists a sequence $(x_n)_{n\in \N} \subset \dom A$ converging to $x_{0}$ such that the strong solution
$x_{k}(\cdot; z_{k})$ of the equation
\begin{equation}
\dot{x}_{k}(t; z_{k})\in f(x(t; z_{k}))-Ax_{k}(t; z_{k}),\text{ \ \ }%
x_{k}(0,z_{k})=z_{k}, \label{equationk}%
\end{equation}
converges   uniformly to $x(\cdot; x_{0})$ on $[0,T].$ Moreover, we have that
\begin{equation}\label{semigroup}
x(s; x(t; x_{0}))=x(s+t; x_{0})\text{ for all }s,t\geq0
\end{equation}
(called the semigroup property). If $L_{f}$ denotes the Lipschitz constant of
$f$ on $\operatorname*{cl}(\dom A),$ then for every $t\geq0$ and
$x_{0},y_{0}\in\operatorname*{cl}(\dom A)$ we have that
\[
\left\Vert x(t; x_{0})-x(t; y_{0})\right\Vert \leq e^{L_{f}t}\left\Vert
x_{0}-y_{0}\right\Vert .
\]

In the remaining part of the paper, $x(\cdot;x_{0})$ denotes the weak solution of Equation
(\ref{id}) (which is also,    a strong one whenever a strong solution exists.)
\vskip 2mm
From now on, we suppose throughout  this section that
\[
\Ovalbox{$
\inte\left(  \operatorname*{co}\left(  \dom %
A\right)  \right)  \neq\emptyset. $\label{conditionint}%
}
\]
Hence,  $\inte\left(  \dom A\right)  $ is convex,
$\inte\left(  \dom A\right)  =\inte%
\left(  \operatorname*{co}\left(  \dom A\right)  \right)
=\inte\left(  \operatorname*{cl}\left(  \dom %
A\right)  \right)  ,$ and $A$ is locally bounded on $\inte%
\left(  \dom A\right)  $. Therefore, a (unique) strong solution of
(\ref{id}) always exists \cite{BrezisBook}. We have the following technical lemma, adding
more information about the qualitative behavior of this solution.

\begin{lemma}
\label{lemarint}Let $\bar{y}\in\dom A\ $and $\rho>0$ be such
that $B_{\rho}(\bar{y})\subset\inte\left(  \operatorname*{co}%
\left(  \dom A\right)  \right)  .$ Then, $M:=\sup_{z\in B_{\rho
}(\bar{y})}\left\Vert (f(z)-Az)^{\circ}\right\Vert <\infty$ and for all $y\in
B_{\rho}(\bar{y})$ and $t\leq1$ we have that
\[
\left\Vert \frac{d^{+}x(\cdot;y)}{dt}(t)\right\Vert \leq e^{L_{f}}M.
\]

\end{lemma}

\begin{proof}
By virtue of the semi-group property (\ref{sgp}), the following inequality
holds for all $y\in\operatorname*{cl}(\dom A)$ and $0\leq t<s$
(e.g., \cite[Lemma 1.1]{BrezisBook})%
\begin{equation}
\left\Vert x(t+s; y)-x(t; y)\right\Vert =\left\Vert
x(t; x(s;y))-x(t; y)\right\Vert \leq e^{L_{f}t}\left\Vert x(s; y)-y\right\Vert .
\label{brr}%
\end{equation}
In particular, for $y\in B_{\rho}(\bar{y})$ and $t\leq1$ we get that
\begin{align*}
\left\Vert \frac{d^{+}x(\cdot; y)}{dt}(t)\right\Vert =\lim_{s\downarrow0}%
s^{-1}\left\Vert x(t+s; y)-x(t; y)\right\Vert  &  \leq e^{L_{f}t}\lim
_{s\downarrow0}s^{-1}\left\Vert x(s;y)-y\right\Vert \\
&  =e^{L_{f}t}\left\Vert \frac{d^{+}x(\cdot;y)}{dt}(0)\right\Vert \\
&  =e^{L_{f}t}\left\Vert (f(y)-Ay)%
{{}^\circ}%
\right\Vert \leq e^{L_{f}}M.
\end{align*}

The fact that $M$ is finite  follows from the maximal monotonicity of $A$
together with the Lipschitz continuity of $f.$\smartqed 
\end{proof}

\begin{definition}
\label{deflyap}Let be given functions $V\in\mathcal{F}(H),$ $W\in
\mathcal{F}(H;\overline{{\mathbb{R}}}_{+})$ and a number $a\in{\mathbb{R}}%
_{+}.$ We say that $(V,W)$ forms a  $\emph{a}$\emph{-}Lyapunov pair for
\emph{(\ref{id}) }with respect to a set $D\subset\operatorname*{cl}%
(\dom A)$ if for all $y\in D$ we have that
\begin{equation}\label{Lyapunovpair}
e^{at}V(x(t; y))+\int_{0}^{t}W(x(\tau;y))d\tau\leq V(y)\text{ for all }t\geq0.
\end{equation}
\end{definition}
\noindent $\emph{a}$\emph{-}Lyapunov pairs with respect to $\operatorname*{cl}%
(\dom A)$ are simply called $\emph{a}$\emph{-}Lyapunov pairs
(see \cite{AHT2012}); in addition, if $a=0$ and $W=0$,  we recover the
classical concept of Lyapunov functions. The case $D=$ $\inte%
(\dom A)$ (when nonmepty), or $D=$ $\operatorname*{ri}%
(\dom A)$ in the finite-dimensional setting, is useful too since
it allows recovering the behaviour of $V$ on the whole set $\operatorname*{cl}%
(\dom A)$ when, as in Proposition \ref{cont} below, some
continuity conditions  on $V$ are known. More precisely, our characterization
theorem, Theorem \ref{heartlocalfinite} below, provides criteria for Lyapunov
pairs with respect to small sets, for instance balls, rather than the whole
set $\inte(\dom A).$
%
The lack of regularity properties of    $\emph{a}$\emph{-}Lyapunov pairs
$(V,W)$ in Definition \ref{deflyap} is mainly due to the non-smoothness of the
function $V.$  Let us remind that
 inequality (\ref{Lyapunovpair})  also
holds if instead of $W$ one considers its Moreau-Yosida regularization, which
is Lipschitz continuous on every bounded subset of $H$. This follows from  the next Lemma \ref{moreau} (e.g  \cite{AHT2012}).
\begin{lemma}
\label{moreau}For every $W\in\mathcal{F}(H;\overline{{\mathbb{R}}}_{+}),$
there exists a sequence of functions $(W_{k})_{k\in\mathbb{N}}\subset
\mathcal{F}(H,{\mathbb{R}}_{+})$ converging to $W$ (for instance,
$W_{k}\uparrow W$) such that each $W_{k}$ is Lipschitz continuous on every
bounded subset of $H,$ and satisfies $V(y)>0$ if and only if $V_{k}%
(y)>0$.
\end{lemma}
Consequently, if $V,$ $D\subset\operatorname*{cl}%
(\dom A),$ and $a\in{\mathbb{R}}_{+}$ are as in Definition \ref{deflyap}  then, with respect to $D$,  $(V,W)$ forms an
$\emph{a}$\emph{-}Lyapunov pair for \emph{(\ref{id}) }if and only if each pair
$(V,W_{k})$ forms an $\emph{a}$\emph{-}Lyapunov pair for \emph{(\ref{id})}.

\begin{proposition}
\label{cont}\bigskip Let be given functions $V\in\mathcal{F}(H),$
$W\in\mathcal{F}(H;\overline{{\mathbb{R}}}_{+})$ and a number $a\in
{\mathbb{R}}_{+}.$ If $V$ verifies
\begin{equation}
\liminf_{\dom A\ni z\rightarrow y}V(z)=V(y)\text{ \ for
all}\ y\in\operatorname*{cl}(\dom A)\cap\dom V,
\label{regular}%
\end{equation}
then it is equivalent to saying  that $(V,W)$ forms an $\emph{a}$\emph{-}Lyapunov
pair with respect to either $\dom A$ or $\operatorname*{cl}%
(\dom A).$
\end{proposition}

Property  (\ref{regular}) has been already used in
\cite{Kocanetall02}, and implicitely in \cite{Pazy81}, among other works. It
holds, if for instance,  $V$  ($\in F(H)$ ) is convex and its
effective domain has a nonempty interior such that $\inte%
(\dom V)\subset\dom A.$

Our starting point is the next result which characterizes %
$a$ -Lyapunov pairs locally in $\inte%
(\dom A).$  The general form corresponding to %
$a$ -Lyapunov pairs in $\operatorname*{cl}(\dom %
A)$  was recently  established in \cite{AHT2012}. For the reader
convenience we include here a sketch of the proof.
\begin{proposition}
 \label{heartlocal}
 Assume that $\inte\left(
\operatorname*{co}\{\dom A\}\right)  \neq\emptyset.$ Let
$V\in F_{w}(H)$  satisfy $\dom V\subset
\operatorname*{cl}(\dom A)$, $W\in F(H;\overline
{{\mathbb{R}}}_{+}),$  and $a\in R_{+}.$  Let  $\bar{y}\in
H,$ $\bar{\lambda}\in\lbrack-\infty,V(\bar{y})),$ and %
$\bar{\rho}\in(0,+\infty]$  be such that %
\[
\dom V\cap B_{\bar{\rho}}(\bar{y})\cap\lbrack V>\bar{\lambda
}]\subset\inte(\dom A).
\]
 Then, the following statements are equivalent:\newline\newline(i)
$\forall y\in\dom V\cap B_{\bar{\rho}}(\bar{y})\cap\lbrack
V>\bar{\lambda}]$ %
\[
\sup_{\xi\in\partial_{P}V(y)}\min_{\upsilon \in Ay}\left\langle \xi
,f(y)-\upsilon \right\rangle +aV(y)+W(y)\leq0;
\]
(ii) $\forall y\in\dom V\cap B_{\bar{\rho}}(\bar
{y})\cap\lbrack V>\bar{\lambda}]$ %
\[
\sup_{\xi\in\partial_{P}V(y)}\left\langle \xi,f(y)-\pi_{Ay}(f(y))\right\rangle
+aV(y)+W(y)\leq0;
\]
(iii) $\forall y\in B_{\bar{\rho}}(\bar{y})\cap\lbrack V>\bar
{\lambda}]$ we have that %
\[
e^{at}V(x(t; y))+\int_{0}^{t}W(x(\tau;y))d\tau\leq V(y)\text{ \ \ }\forall
t\in\left[  0,\rho(y)\right]  ,
\]
where %
\begin{equation}
\rho(y):=\sup\left\{  \nu>0\left\vert
\begin{array}
[c]{l}%
\exists\rho>0~s.t.~B_{\rho}(y)\subset B_{\bar{\rho}}(\bar{y})\cap\lbrack
V>\bar{\lambda}],\text{ and for all }t\in\lbrack0,\nu]\\
~2\left\Vert x(t;y)-y\right\Vert <\frac{\rho}{2}\text{ and }\\
\left\vert (e^{-at}-1)V(y)-\int_{0}^{t}W(x(\tau;y))d\tau\right\vert
<\frac{\rho}{2}\text{ }%
\end{array}
\right.  \right\}  .\label{rhox0}%
\end{equation}

\end{proposition}

\begin{remark}
(Before the proof) the constant $\rho(y)$ defined in
(\ref{rhox0}) is positive whenever $y\in\operatorname*{cl}%
(\dom A)\cap B_{\bar{\rho}}(\bar{y})\cap\lbrack V>\bar{\lambda
}].$ Hence, when $\bar{\rho}=-\bar{\lambda}=\infty$ one can
easily show that (iii) is equivalent to (see \cite[Proposition 3.2]{AHT2012}) %
\[
e^{at}V(x(t; y))+\int_{0}^{t}W(x(\tau;y))d\tau\leq V(y)\emph{\ \ }%
\text{\ \ \ \ }\emph{for\ all\ }t\geq0;
\]
that is, $(V,W)$ forms a Lyapunov pair with respect to
$\inte(\dom A).$
\end{remark}

\begin{proof}
For simplicity,  we suppose that  $W\equiv0.$

(iii) $\implies$ (ii) Let us fix $y\in B_{\bar{\rho}}(\bar
{y})\cap\lbrack V>\bar{\lambda}]$ and $\xi\in\partial_{P}%
V(y)$  so that  $y\in B_{\bar{\rho}}(\bar{y})\cap\lbrack V>\bar
{\lambda}]\cap\dom V\subset\dom A$ and
there exist  $\alpha>0$  and  $T\in(0,\rho(y))$  such that %
\[
\langle\xi,x(t;y)-y\rangle\leq V(x(t;y))-V(y)+\alpha\left\Vert
x(t;y)-y\right\Vert ^{2}\leq\alpha\left\Vert x(t;y)-y\right\Vert ^{2}\text{
for all }t\in\lbrack0,T).
\]
 But  $y\in\dom A$  and so there exists a constant
$l\geq0$  such that %
\[
\langle\xi,t^{-1}(x(t;y)-y)\rangle\leq l\left\Vert x(t;y)-y\right\Vert \text{
for all }t\in\lbrack0,T);
\]
 hence, taking the limit as $t\rightarrow0^{+}$ we obtain that
\[
\langle\xi,f(y)-\pi_{Ay}(f(y))\rangle\leq0;
\]
that is, (ii) follows.

 (i) $\implies$ (iii) To simplify the proof of this part, we
assume that   $f\equiv0,$   $W$  $\equiv0$  and %
$a=0.$  For this aim we fix  $y\in\dom V\cap B_{\bar
{\rho}}(\bar{y})\cap\lbrack V>\bar{\lambda}]$  and let $\rho
>0$  and $v>0$  be such that %
\begin{equation}
B_{\rho}(y)\subset B_{\bar{\rho}}(\bar{y})\cap\lbrack V>\bar{\lambda}]\text{
\ and }\label{f0}%
\end{equation}%
\begin{equation}
\sup_{t\in\lbrack0,\nu]}2\left\Vert x(t;y)-y\right\Vert <\rho;\label{rhobar}%
\end{equation}
 the existence of such scalars  $\rho$  and $v$  is a
consequence of the lower semicontinuity  of  $V$  and the Lipschitz continuity of $x(\cdot;\cdot)$  (see Lemma \ref{lemarint}). Let   $T<\nu$  be
fixed and define the functions  $z(\cdot):[0,T]\subset R_{+}\rightarrow
H\times R$  and  $\eta(\cdot):[0,T]\subset R_{+}\rightarrow R_{+}%
$  as %
\begin{equation}
z(t):=\left(  x(t;y),V(y)\right)  ,\text{ }\eta(t):=\frac{1}{2}d^{2}%
(z(t),\operatorname*{epi}V);\label{hf1}%
\end{equation}
 observe that  $z(\cdot)$  and $\eta(\cdot)$  are
Lipschitz continuous on $[0,T).$  Now, using a standard chain rule
(e.g. \cite{ClarkeBook}), for fixed  $t\in(0,T)$  it holds that %
\[
\partial_{C}\eta(t)=d(z(t),\operatorname*{epi}V)\partial_{C}d(z(\cdot
),\operatorname*{epi}V)(t).
\]
 So, from one hand we get  $\partial_{C}\eta(t)=\{\theta\}$ 
whenever $z(t)\in\operatorname*{epi}V.$  On the other hand, when
$z(t)\not \in \operatorname*{epi}V$  we obtain that 
\begin{equation}
\partial_{C}\eta(t)\subset\overline{\operatorname*{co}}\left[  \bigcup
\limits_{(u,\mu)\in\Pi_{\operatorname*{epi}V}(z(t)),\text{ }u\in B_{\rho}%
(y)}\left\langle x(t;y)-u,-Ax(t;y)\right\rangle \right]  ;\label{cc}%
\end{equation}
 the fact that  $u\in B_{\rho}(y)$  is a consequence of the
following inequalities:%
\begin{align*}
\left\Vert u-y\right\Vert  &  \leq\left\Vert x(t;y)-u\right\Vert +\left\Vert
x(t;y)-y\right\Vert \\
&  \leq\left\Vert (x(t;y),V(y))-(u,\mu)\right\Vert +\left\Vert
x(t;y)-y\right\Vert \\
&  \leq\left\Vert (x(t;y),V(y))-(y,V(y))\right\Vert +\left\Vert
x(t;y)-y\right\Vert \text{ }\\
&\leq2\left\Vert x(t;y)-y\right\Vert <\rho.
\end{align*}
 Take now $\xi\in Ax(t;y)$  and $(u,\mu)\in
\Pi_{\operatorname*{epi}V}(z(t))$  with $u\in B_{\rho}(y)$  so
that  $V(y)-\mu\leq0$  and $u\in\dom V\cap B_{\bar{\rho
}}(\bar{y})\cap\lbrack V>\bar{\lambda}]$  (recall (\ref{f0})).
\vskip 2mm
 If
$V(y)-\mu<0,$  we write  $(\mu-V(y))^{-1}$  $(x(t;y)-u)\in
\partial V_{P}(u).$  Then, by the current assumption (i),  select 
$\upsilon\in Au$ such that %
\[
\left\langle (\mu-V(y))^{-1}(x(t;y)-u),-\upsilon \right\rangle \leq0.
\]
 Therefore,   invoking the monotonicty of $A$  we get %
\[
\langle x(t;y)-u,-\vartheta \rangle=\langle x(t;y)-u,-\upsilon \rangle+\langle
x(t;y)-u,\upsilon -\vartheta \rangle\leq\langle x(t;y)-u,-\upsilon \rangle\leq0.
\]
 Since  $\xi\in Ax(t;y)$  is arbitrary and according to (\ref{cc}), \ we deduce that  $\partial_{C}\eta(t)\subset \R_{-}.$  \vskip 2mm
 If
$V(y)-\mu=0$  so that $(x(t;y)-u,0)\in N_{\operatorname*{epi}%
V}(u,V(u))$  and $x(t;y)-u\neq\theta.$  Then, for a fixed
$\varepsilon>0$  such that  $B_{\varepsilon}(u)\subset B_{\rho}%
(y)\cap\inte(\dom (A))$  (recall that
$u\in B_{\rho}(y)\cap\dom V\cap B_{\bar{\rho}}(\bar{y}%
)\cap\lbrack V>\bar{\lambda}]\subset\inte(\dom %
(A))$ ), take   $u_{\varepsilon}\in B_{\varepsilon}(u)\cap
\dom V$  with $\left\vert V(u)-V(u_{\varepsilon
})\right\vert \leq\varepsilon,$   $\alpha\in(0,\varepsilon)$ 
and  $\xi\in B_{\varepsilon}(x(t;y)-u)$   such that $\alpha^{-1}%
\xi\in\partial V_{P}(u_{\varepsilon})$  (see, e.g., \cite[Theorem
2.4]{ClarkeBook}). Hence, using the current assumption,  select   $\xi_{\varepsilon} \in Au_{\varepsilon}$  such that  $\left\langle \xi
,-\xi_{\varepsilon} \right\rangle \leq\alpha\varepsilon$.   Hence,
\[
\langle x(t;y)-u,-\xi_{\varepsilon} \rangle\leq\varepsilon\left\Vert
\xi_{\varepsilon} \right\Vert +\left\langle \xi,-u_{\varepsilon}^{\ast
}\right\rangle \leq\varepsilon\left\Vert \xi_{\varepsilon} \right\Vert
+\alpha\varepsilon\leq\varepsilon\left\Vert \xi_{\varepsilon} \right\Vert
+\varepsilon^{2}.
\]
By the monotonicity of  $A$  this yields%
\begin{align*}
\langle x(t;y)-u,-\vartheta \rangle &  \leq\langle x(t;y)-u_{\varepsilon
},-\vartheta \rangle+\varepsilon\left\Vert \vartheta \right\Vert \\
&  \leq\langle x(t;y)-u_{\varepsilon},-\xi_{\varepsilon} \rangle
+\varepsilon\left\Vert \vartheta \right\Vert \\
&  \leq\langle x(t;y)-u,-\xi_{\varepsilon} \rangle+\left\Vert
u_{\varepsilon}-u\right\Vert \left\Vert \xi_{\varepsilon} \right\Vert
+\varepsilon\left\Vert \vartheta \right\Vert \\
&  \leq2\varepsilon\left\Vert \xi_{\varepsilon} \right\Vert +\varepsilon
\left\Vert \vartheta \right\Vert +\varepsilon^{2}.
\end{align*}
 Moreover, as $(u_{\varepsilon})_{\varepsilon\leq1}$  is
bounded in $\inte(\dom (A)),$  the net
$(\xi_{\varepsilon} )_{\varepsilon}$  is also bounded and 
passing to the limit as  $\varepsilon$  goes to $0$  we get %
\[
\langle x(t;y)-u,-\vartheta \rangle\leq0.
\]
 This gives the desired inclusion $\partial_{C}\eta(t)\subset \R_{-}%
$  (recall (\ref{cc})) and so establishes  the proof of (iii).\smartqed 
\end{proof}

 \bigskip

 We are now ready to give the main result of this section, which
provides   a precise improvement of Proposition \ref{heartlocal}.

\begin{theorem}
 \label{heartlocalfinite}Assume that  $\inte\left(
\operatorname*{co}\{\dom A\}\right)  \neq\emptyset.$  Let
$V\in F_{w}(H)$  with  $\inf V>-\infty,$   $W\in F(H;\overline
{{\mathbb{R}}}_{+}),$  and  $a\in R_{+}$  be given. Fix %
$\bar{y}\in\dom V,$   $\bar{\lambda}\in(-\infty,V(\bar
{y}))$  and let  $\bar{\rho}>0$  be such that %
\[
\dom V\cap\lbrack V>\bar{\lambda}]\cap B_{\bar{\rho}}(\bar
{y})\subset\inte(\dom A).
\]
 Then, the following statements are equivalent:\newline\newline(i)
$\forall y\in\dom V\cap B_{\bar{\rho}}(\bar{y})\cap\lbrack
V>\bar{\lambda}]$   %
\[
\sup_{\xi\in\partial_{P}V(y)}\min_{\upsilon \in Ay}\left\langle \xi
,f(y)-\upsilon \right\rangle +aV(y)+W(y)\leq0;
\]
 (ii) (If  $V$  is weakly continuous when restricted to
$B_{\rho}(\bar{y}))$  $\forall y\in\dom V\cap
B_{\bar{\rho}}(\bar{y})\cap\lbrack V>\bar{\lambda}]$  %
\[
e^{at}V(x(t; y))+\int_{0}^{t}W(x(\tau; y))d\tau\leq V(y)\text{ \ for all }%
t\geq0.
\]
 \newline\quad\newline Consequently, if (i)-(ii) holds on
$\inte(\dom A),$  the pair %
$(V,W)$  is an  $a$ -Lyapunov pair for (\ref{id}) with respect
to $\operatorname*{cl}(\dom A).$
\end{theorem}

\begin{proof}
The consequence is immediate once we prove the main conclusion.

First, invoking  Lemma \ref{moreau}
we may assume w.l.o.g. that $W$ is
Lipschitz continuous on every bounded subset of $H.$ In the rest of the proof,
we take $\hat{y}$ in $\dom V\cap B_{\bar{\rho}}(\bar{y}%
)\cap\lbrack V>\bar{\lambda}]$ ($\subset\inte%
(\dom A)$) and, taking into account the lsc of $V,$ choose
$\rho>0$ such that $B_{2\rho}(\hat{y})\subset B_{\bar{\rho}}(\bar{y}%
)\cap\lbrack V>\bar{\lambda}]\cap\inte(\dom A)$
and
\begin{equation}
V(z)\geq V(\hat{y})-1\text{ \ \ }\forall z\in B_{2\rho}(\tilde{y}).
\label{mak}%
\end{equation}
Also, by virtue of Lemma \ref{lemarint}, we consider a positive constant $M$
such that, for all $0\leq t\leq1$ and all $z\in B_{2\rho}(\hat{y}),$
\begin{equation}
\left\Vert \frac{d^{+}x(\cdot; z)}{dt}(t)\right\Vert \leq M; \label{pr}%
\end{equation}
hence, $\left\Vert x(t; z)-z\right\Vert \leq Mt$ and so, by (\ref{mak}),
\begin{equation}
V(x(t; z))\geq V(\hat{y})-1\geq\bar{\lambda}-1\text{ \ \ }\forall z\in B_{\rho
}(\hat{y})\text{ and }\forall t\in\left[  0,\frac{\rho}{M}\right]  .
\label{vlsc1}%
\end{equation}
Let us fix $\gamma\geq1$ and define the set
\begin{equation}
G(\hat{y}):=[\left\vert V\right\vert \leq\left\vert V(\hat{y})\right\vert
+\gamma]. \label{gyb}%
\end{equation}

\textbf{Claim}: there exists $T>0$ such that
\begin{equation}
e^{at}V(x(t; y))+\int_{0}^{t}W(x(\tau; y))d\tau\leq V(y)\text{
\ \ \ \ \ }\forall y\in B_{\rho}(\hat{y})\cap G(\hat
{y}),\text{ }\forall t\in\lbrack0,T]. \label{ma}%
\end{equation}
Using the ($L_{W}$-)Lipschitz continuity of $W$ on the (bounded) set
$\{x(t; y)\mid0\leq t\leq1,y\in B_{2\rho}(\hat{y})\},$ we write, for all $y\in
B_{2\rho}(\hat{y})\cap G(\hat{y})\cap\dom V$ and $0\leq t\leq1,$%
\begin{multline*}
2\left\Vert x(t; y)-y\right\Vert +\left\vert \left(  e^{-at}-1\right)
V(y)-\int_{0}^{t}W(x(\tau;y))d\tau\right\vert \\
\leq2Mt+\left(  1-e^{-at}\right)  (\left\vert V(\hat{y})\right\vert
+\gamma)+(W(\hat{y})+L_{W}(M+2\rho))t.
\end{multline*}
Therefore, we can choose $T>0$ so that for all $y\in B_{2\rho}(\hat{y})\cap
G(\hat{y})$ we have that
\[
\sup_{t\in\lbrack0,T]}2\left\Vert x(t; y)-y\right\Vert +\left\vert \left(
e^{-at}-1\right)  V(y)-\int_{0}^{t}W(x(\tau;y))d\tau\right\vert <\frac{\rho
}{2}.
\]
We also observe that for any given $y\in B_{\rho}(\hat{y})\cap G(\hat{y})$ we
have that $B_{\rho}(y)\subset B_{2\rho}(\hat{y})\cap\lbrack V>\bar{\lambda}].$
Therefore, since
\[
B_{\rho}(\hat{y})\cap G(\hat{y})\subset B_{2\rho}(\hat{y})\cap\lbrack
V>\bar{\lambda}]\subset B_{\bar{\rho}}(\bar{y})\cap\lbrack V>\bar{\lambda
}]\cap\dom V,
\]
the claim follows from  Theorem \ref{heartlocal}.

To go further in the proof, we fix two parameters $\varepsilon,\delta>0$ and we 
introduce the set $E_{\varepsilon,\delta}\subset\mathbb{R}_{+}$ defined as
\[
E_{\varepsilon,\delta}:=\left\{  \lambda\in\mathbb{R}_{+}\left\vert
\begin{array}
[c]{l}%
\exists\rho_{1},\rho_{2}\in(\frac{\rho}{2},\rho),\rho_{1}<\rho_{2},\exists
\rho_{\lambda}\in(\frac{\rho}{2},\rho_{2}),\text{ }\forall y\in B_{\rho
_{\lambda}}(\hat{y})\cap G(\hat{y}),\text{ }\forall t\leq\lambda:\\
\\
e^{at}V_{\delta}(x(t; y))+\int_{0}^{t}W(x(\tau;y))d\tau\leq V(y)+\varepsilon
(\rho_{1}-\frac{\rho}{2})(\rho-\rho_{2})
\end{array}
\right.  \right\}  ,
\]
where $V_{\delta}:H\rightarrow\mathbb{R}$ is the function given by
\[
V_{\delta}(y):=\inf_{z\in H}\{ V(z)+\frac{1}{\delta}\left\Vert y-z\right\Vert ^{2}\}.
\]
$V_{\delta}$  is dominated by $V$ and  is Lipschitz continuous on the bounded sets of $H.$
 Then, we have
that $[0,T]\subset E_{\varepsilon,\delta}, $ that is, $E_{\varepsilon,\delta
}\neq\emptyset.$ Next, we shall show that $E_{\varepsilon,\delta}%
=\mathbb{R}_{+}$ or, equivalentely, that $E_{\varepsilon,\delta}$ is closed
and open with respect to the usual topolgy on $\mathbb{R}_{+}.$

\textbf{Claim}: $E_{\varepsilon,\delta}$ is closed.

\noindent Let a sequence $(\lambda_{n})_{n\in\N}\subset E_{\varepsilon,\delta}$ be such that
$\lambda_{n}\rightarrow\tilde{\lambda}$ and, by the definition of
$E_{\varepsilon,\delta},$ take $(\rho_{1,n})_{n\in\N},(\rho_{2,n})_{n\in\N},$\newline
$(\rho
_{3,n})_{n\in\N}\subset(\frac{\rho}{2},\rho)$ be such that
\[
\rho_{1,n}<\rho_{2,n},\text{ }\rho_{3,n}\in(\frac{\rho}{2},\rho_{2,n}),
\]
together with the relation 
\[
e^{at}V_{\delta}(x(t; y))+\int_{0}^{t}W(x(\tau;y))d\tau\leq V(y)+\varepsilon
(\rho_{1,n}-\frac{\rho}{2})(\rho-\rho_{2,n})\text{ \ }\ \text{\ }\forall y\in
B_{\rho_{3,n}}(\hat{y})\cap G(\hat{y}),
\]
valid for all $ t\leq\lambda_{n}$.
Because all the sequences $(\rho_{1,n})_{n\in \N},(\rho_{2,n})_{n\in \N},$ and
$(\rho_{3,n})_{n\in\N}$ are bounded, on relabeling  if necessary, we may
suppose that $\rho_{1,n}\rightarrow\rho_{1}\in\lbrack\frac{\rho}{2},\rho]$,
$\rho_{2,n}\rightarrow\rho_{2}\in\lbrack\frac{\rho}{2},\rho_{2}],$ and
$\rho_{3,n}\rightarrow\hat{\rho}\in\lbrack\frac{\rho}{2},\rho_{2}].$ As well,
it is enough to suppose that $\tilde{\lambda}>$ $T$ and $\tilde{\lambda
}>\lambda_{n}$ for all $n$ because, otherwise, either $\tilde{\lambda}%
\leq\lambda_{n}$ for some $n$ or $\tilde{\lambda}\leq T;$ hence in both cases
we have  $\tilde{\lambda}\in E_{\varepsilon,\delta}.$

If $y\in B_{\hat{\rho}}(\hat{y})\cap G(\hat{y})$ and $t<\tilde{\lambda},$ for
all $n$ large enough we get that $y\in B_{\rho_{n}}(\hat{y})\cap G(\hat{y})$ and
$t<\lambda_{n}$ and, so,
\[
e^{at}V_{\delta}(x(t; y))+\int_{0}^{t}W(x(\tau;y))d\tau\leq V(y)+\varepsilon
(\rho_{1,n}-\frac{\rho}{2})(\rho-\rho_{2,n}).
\]
As $n$ goes to $\infty$ we obtain that
\[
e^{at}V_{\delta}(x(t; y))+\int_{0}^{t}W(x(\tau;y))d\tau\leq V(y)+\varepsilon
(\rho_{1}-\frac{\rho}{2})(\rho-\rho_{2});
\]
this inequality also holds for $t=\tilde{\lambda}$ in view of the continuity
of $V_{\delta}.$ It is also useful to notice here that for all $y\in
\overline{B}_{\frac{\rho}{2}}(\hat{y})\cap G(\hat{y})$ and $t\leq
\tilde{\lambda}$
\begin{equation}
e^{at}V_{\delta}(x(t; y))+\int_{0}^{t}W(x(\tau;y))d\tau\leq V(y). \label{vn}%
\end{equation}

Now,  checking the possible values of $\rho_{1},\rho_{2},$ and $\hat{\rho}%
$\ we observe that only two cases may occur: the first corresponds to
$(\rho_{1}-\frac{\rho}{2})(\rho-\rho_{2})=0$ and happens when $\rho_{1}%
=\frac{\rho}{2},$ $\rho_{2}=\rho,$ or $\rho_{2}=\rho;$ this last equality
implies that $\frac{\rho}{2}\leq\rho_{1}\leq\rho_{2}\leq\frac{\rho}{2}$ and,
so, $(\rho_{1}-\frac{\rho}{2})(\rho-\rho_{2})=0.$ While the second case
corresponds to $(\rho_{1}-\frac{\rho}{2})(\rho-\rho_{2})>0$ and happens when
$\rho_{1},\rho_{2}\in(\frac{\rho}{2},\rho).$

To  begin with, we   analyze the case $(\rho_{1}-\frac{\rho}{2})(\rho-\rho_{2})>0.$
This necessarily implies that $\rho_{2}>\frac{\rho}{2}$ in view of the
inequality $\frac{\rho}{2}\leq\rho_{1}\leq\rho_{2}$. We may suppose that
$\hat{\rho}=\frac{\rho}{2}$ because otherwise $\hat{\rho}\in(\frac{\rho}%
{2},\rho_{2})$ trivially yields $\tilde{\lambda}\in E_{\varepsilon,\delta}.$
So, in order to prove that $\tilde{\lambda}\in E_{\varepsilon,\delta},$ we
only need to find some $\beta>0$ such that $\frac{\rho}{2}+\beta\in(\frac
{\rho}{2},\rho_{2})$ and for all $y\in B_{\frac{\rho}{2}+\beta}(\hat{y})\cap
G(\hat{y})$ and $t\leq\tilde{\lambda},$%
\begin{equation}
e^{at}V_{\delta}(x(t; y))+\int_{0}^{t}W(x(\tau; y))d\tau\leq V(y)+\varepsilon
(\rho_{1}-\frac{\rho}{2})(\rho-\rho_{2}). \label{asd}%
\end{equation}
Proceeding by contradiction, we assume that for each $k\geq1$ verifying
$\frac{\rho}{2}+\frac{1}{k}\in(\frac{\rho}{2},\rho_{2}),$ there exist
$y_{k}\in B_{\frac{\rho}{2}+\frac{1}{k}}(\hat{y})\cap G(\hat{y})$ and
$0<t_{k}\leq\tilde{\lambda}$ such that
\begin{equation}
e^{at_{k}}V_{\delta}(x(t_{k};y_{k}))+\int_{0}^{t_{k}}W(x(\tau;y_{k}%
))d\tau>V(y_{k})+\varepsilon(\rho_{1}-\frac{\rho}{2})(\rho-\rho_{2}).
\label{hg}%
\end{equation}
Because of (\ref{vn}) we must have $(y_{k})_{k}\subset B_{\frac{\rho}{2}%
+\frac{1}{k}}(\hat{y})\setminus\overline{B}_{\frac{\rho}{2}}(\hat{y}).$
W.l.o.g. we may suppose that $t_{k}\rightarrow\tilde{t}\leq\tilde{\lambda}.$
For each $k,$ we denote by $\tilde{y}_{k}\in\overline{B}_{\frac{\rho}{2}}%
(\hat{y})$ the orthogonal projection of $y_{k}$ onto $\overline{B}_{\frac
{\rho}{2}}(\hat{y}).$ Thus, from one hand, we may also suppose that
$(\tilde{y}_{k})_{k}$ weakly converges to some $\tilde{y}\in\overline
{B}_{\frac{\rho}{2}}(\hat{y}).$ Furthermore, from the inequality $\left\Vert
y_{k}-\tilde{y}_{k}\right\Vert \leq\frac{1}{k}$ we infer that $y_{k}$ also
weakly converges to $\tilde{y}$ and, so, by the weak continuity of $V$ on
$B_{\rho}(\hat{y}),$%
\begin{equation}
V(\tilde{y})=\lim_{k}V(\tilde{y}_{k})=\lim_{k}V(y_{k}). \label{wls}%
\end{equation}
Hence,
\[
\left\vert V(\tilde{y})\right\vert =\lim_{k}\left\vert V(\tilde{y}%
_{k})\right\vert =\lim_{k}\left\vert V(y_{k})\right\vert \leq\left\vert
V(\hat{y})\right\vert +1.
\]
In particular, (w.l.o.g.) this implies that $$(\tilde{y}_{k})_{k}\cup
\{\tilde{y}\}\subset\overline{B}_{\frac{\rho}{2}}(\hat{y})\cap\lbrack
\left\vert V\right\vert \leq\left\vert V(\hat{y})\right\vert +1]=\overline
{B}_{\frac{\rho}{2}}(\hat{y})\cap G(\hat{y}).$$
On the other hand,   the absolute continuity of $x(\cdot; \tilde{y}_{k})$ yields\[
x(t_{k}; \tilde{y}_{k})-\tilde{y}_{k}=\int_{0}^{t_{k}}\dot{x}(\tau;\tilde
{y}_{k})d\tau
\]
and,   since that $\dot{x}(\cdot; \tilde{y}_{k})\in L^{\infty}([0,\tilde
{\lambda}];H),$ the following holds:
\begin{align*}
\left\Vert x(t_{k};\tilde{y}_{k})-\tilde{y}_{k}\right\Vert \leq t_{k}%
\sup_{\tau\in\lbrack0,t_{k}]}\left\Vert \dot{x}(\tau;\tilde{y}_{k}%
)\right\Vert  &  \leq t_{k}\sup_{\tau\in\lbrack0,t_{k}]}e^{L_{f}\tau
}\left\Vert (f(\tilde{y}_{k})-A\tilde{y}_{k})^{\circ}\right\Vert \\
&  \leq\tilde{\lambda}e^{L_{f}\tilde{\lambda}}\sup_{z\in\overline{B}%
_{\frac{\rho}{2}}(\hat{y})}e^{L_{f}\tau}\left\Vert (f(z)-Az)^{\circ
}\right\Vert \leq M\tilde{\lambda}e^{L_{f}\tilde{\lambda}}.
\end{align*}
Hence, w.l.o.g. we may suppose that the bounded sequence $(x(t_{k}; \tilde
{y}_{k}))_{k\in\N}$ weakly converges in $H.$ Furthermore,  the inequality
\[
\left\Vert x(t_{k};y_{k})-x(t_{k};\tilde{y}_{k})\right\Vert \leq e^{L_{f}%
t_{k}}\left\Vert y_{k}-\tilde{y}_{k}\right\Vert \leq\frac{e^{L_{f}\tilde{t}}%
}{k},
\]
infers that the both   sequences $(x(t_{k};y_{k}))_{k\in\N}$ and $(x(t_{k};\tilde{y}_{k}%
))_{k\in\N}$ weakly converge to the same point in $H.$

On another hand, since the sequences  $(x(t_{k};y_{k}))_{k\in\N}$ and $(x(t_{k};\tilde
{y}_{k}))_{k\in\N}$ are bounded, there exits some $l\geq0$ such that for
all $t\leq\tilde{t}$%
\begin{align*}
\left\vert W(x(t; y_{k}))-W(x(t; \tilde{y}_{k}))\right\vert +\left\vert
V_{\delta}(x(t; y_{k}))-V_{\delta}(x(t; \tilde{y}_{k}))\right\vert  &  \leq
l\left\Vert x(t; y_{k})-x(t; \tilde{y}_{k})\right\Vert \\
&  \leq\frac{le^{L_{f}\tilde{t}}}{k}%
\end{align*}
and, so, we deduce that (w.l.o.g.)%
\begin{equation}
\lim_{k}V_{\delta}(x(t_{k};y_{k}))=\lim_{k}V_{\delta}(x(t_{k};\tilde{y}%
_{k}))\text{ and }\lim_{k}W(x(t_{k};y_{k}))=\lim_{k}W(x(t_{k};\tilde{y}_{k}).
\label{lli}%
\end{equation}
Using Lebesgue's Theorem, this infers
\begin{align*}
&\lim_{k}\left[  e^{a\tilde{t}}V_{\delta}(x(t_{k};\tilde{y}_{k}))+\int_{0}^{\tilde{t}}W(x(\tau;\tilde{y}_{k}))d\tau\right] 
\\& \qquad\qquad\qquad\qquad\ \ =e^{a\tilde{t}}\lim_{k}V_{\delta}(x(t_{k};\tilde{y}_{k}))+\int_{0}^{\tilde{t}}\lim_{k}W(x(\tau;\tilde{y}_{k}))d\tau\\
&\qquad\qquad\qquad\qquad\ \  =e^{a\tilde{t}}\lim_{k}V_{\delta}(x(t_{k};y_{k}))+\int_{0}^{\tilde{t}}\lim_{k}W(x(\tau;y_{k}))d\tau\\
&\qquad\qquad\qquad\qquad\ \  =\lim_{k}\left[  e^{at_{k}}V_{\delta}(x(t_{k};y_{k}))+\int_{0}^{t_{k}}W(x(\tau;y_{k}))d\tau\right] . 
\end{align*}
Consequently,  taking limits in (\ref{hg}), and using (\ref{wls}) we obtain
\begin{align} \label{golf}
\lim_{k}\Biggl[  e^{a\tilde{t}}V_{\delta}(x(t_{k};\tilde{y}_{k}))&+\int%
_{0}^{\tilde{t}}W(x(\tau;\tilde{y}_{k}))d\tau\Biggr] \\
&\geq\lim_{k}%
V(y_{k})+\varepsilon(\rho_{1}-\frac{\rho}{2})(\rho-\rho_{2})\nonumber\\
&=V(\tilde{y})+\varepsilon(\rho_{1}-\frac{\rho}{2})(\rho-\rho_{2}%
)\nonumber\\
&=V(\lim_{k}\tilde{y}_{k})+\varepsilon(\rho_{1}-\frac{\rho}{2})(\rho
-\rho_{2})\nonumber\\
&=\lim_{k}V(\tilde{y}_{k})+\varepsilon(\rho_{1}-\frac{\rho}{2})(\rho
-\rho_{2}))\nonumber.%
\end{align}
In other words, for $k$ large enough we have  
\[
e^{a\tilde{t}}V_{\delta}(x(t_{k};\tilde{y}_{k}))+\int_{0}^{\tilde{t}}%
W(x(\tau;\tilde{y}_{k}))d\tau\geq V(\tilde{y}_{k})+\frac{\varepsilon}{2}%
(\rho_{1}-\frac{\rho}{2})(\rho-\rho_{2})>V(\tilde{y}_{k}),
\]
and a contradiction to  (\ref{vn}) as $\tilde{y}_{k}\in\overline
{B}_{\frac{\rho}{2}}(\hat{y})\cap G(\hat{y}),$ and $\tilde{t}\leq
\tilde{\lambda}.$ Hence, we conclude that some $\rho_{\tilde{\lambda}}%
\in(\frac{\rho}{2},\rho_{2})$ exists so that (\ref{asd}) holds for all $y\in
B_{\frac{\rho}{2}+\beta}(\hat{y})\cap G(\hat{y})$ and $t\leq\tilde{\lambda}.$
This fact shows that $\tilde{\lambda}\in E_{\varepsilon,\delta}.$

It remains to analyse the other case corresponding to $(\rho_{1}-\frac{\rho
}{2})(\rho-\rho_{2})=0.$ If this happens, we choose $\tilde{\rho}_{1}%
,\tilde{\rho}_{2}\in(\frac{\rho}{2},\rho)$ such that $\tilde{\rho}_{1}%
<\tilde{\rho}_{2}$ and $(\rho_{1}-\frac{\rho}{2})(\rho-\rho_{2})>0.$ Thus,
following the same argument as in the first case,  taking into account
(\ref{vn}) we can find some $\beta>0,$ with $\hat{\rho}+\beta\in(\frac{\rho
}{2},\rho_{2}),$ so that (\ref{asd}) holds for all $y\in B_{\frac{\rho}%
{2}+\beta}(\hat{y})\cap G(\hat{y}).$ This shows that $\tilde{\lambda}\in
E_{\varepsilon,\delta}$ and, hence, establishes  the proof of the closedness of
$E_{\varepsilon,\delta}.$

\textbf{Claim}: $E_{\varepsilon,\delta}$ is open. 

\noindent Fix
$\lambda\in E_{\varepsilon,\delta}$ (it is sufficient to take $\lambda\geq
\nu>0$), and let $\rho_{1},\rho_{2}\in(\frac{\rho}{2},\rho)$ and $\hat{\rho
}\in(\frac{\rho}{2},\rho_{2})$\ be such that $\rho_{1}<\rho_{2}$ and, for all
$y\in B_{\hat{\rho}}(\hat{y})\cap G(\hat{y})$ and $t\leq\lambda,$
\begin{equation}
e^{at}V_{\delta}(x(t; y))+\int_{0}^{t}W(x(\tau;y))d\tau\leq V(y)+\varepsilon
(\rho_{1}-\frac{\rho}{2})(\rho-\rho_{2}). \label{nm}%
\end{equation}
We let $\hat{\nu}>0$ verify
\[
\hat{\nu}\leq\min\{\nu,\lambda\},\text{ }\frac{\rho}{2}<\hat{\rho}-M\hat{\nu
}<\rho_{2},\text{ and }\frac{\rho}{2}<e^{a\hat{\nu}}\rho_{1}<\rho_{2}.
\]
So, from one hand, for all $0\leq\alpha\leq\hat{\nu}$ and $y\in B_{\hat{\rho
}-M\hat{\nu}}(\hat{y})\cap G(\hat{y})$ it holds, by Lemma \ref{lemarint},
\begin{equation}
\left\Vert x(\alpha; y)-\hat{y}\right\Vert \leq M\alpha+\left\Vert y-\hat
{y}\right\Vert <M\alpha+\hat{\rho}-M\hat{\nu}\leq\hat{\rho}, \label{vlsc}%
\end{equation}
where $M\geq0$ is defined in (\ref{pr}). Hence, by the choice of $\nu$
($\nu\leq\lambda$)$,$ from (\ref{nm}) we infer that $$V(x(\alpha; y))\leq
V(y)\leq\left\vert V(\hat{y})\right\vert +1.$$ Thus,  taking into account
(\ref{vlsc1}) we obtain that
\[
x(\alpha; y)\in B_{\hat{\rho}}(\hat{y})\cap G(\hat{y}).
\]
Now fix $y\in B_{\hat{\rho}-M\hat{\nu}}(\hat{y})\cap G(\hat{y})$ and
$t\in\lbrack0,\lambda].$ From above we have that
\begin{equation}
x(\hat{\nu},y)\in B_{\hat{\rho}}(\hat{y})\cap G(\hat{y}) \label{we}%
\end{equation}
and, so,  applying (\ref{nm}) we get that
\[
e^{at}V_{\delta}(x(t; x(\hat{\nu},y)))+\int_{0}^{t}W(x(\tau;x(\hat{\nu
},y)))d\tau\leq V(x(\hat{\nu},y))+\varepsilon(\rho_{1}-\frac{\rho}{2}%
)(\rho-\rho_{2}).
\]
Thus,  using the semi-group property together with (\ref{nm}) and
(\ref{we}), we infer that
\begin{align*}
e^{a(\hat{\nu}+t)}&V_{\delta}(x(\hat{\nu}+t,y))+\int_{0}^{\hat{\nu}+t}%
W(x(\tau;y))d\tau\\
& Ê=e^{a\hat{\nu}}e^{at}V_{\delta}(x(t; x(\hat{\nu}%
,y)))+\int_{0}^{t}W(x(\tau;x(\hat{\nu},y)))d\tau
+\int_{0}^{\hat{\nu}}W(x(\tau;y))d\tau\\
& Ê\leq e^{a\hat{\nu}}\left[ Êe^{at}V_{\delta}(x(t; x(\hat{\nu},y)))+\int%
_{0}^{t}W(x(\tau;x(\hat{\nu},y)))d\tau\right]+\int_{0}^{\hat{\nu}}W(x(\tau;y))d\tau\\
& Ê\leq e^{a\hat{\nu}}V(x(\hat{\nu},y))+\int_{0}^{\hat{\nu}}W(x(\tau
,y))d\tau+\varepsilon e^{a\hat{\nu}}(\rho_{1}-\frac{\rho}{2})(\rho-\rho_{2}).
\end{align*}
At this step, for the choice that we made on $\hat{\nu}$ ($\hat{\nu}\leq\nu),$
the last inequality above reads, for all $y\in B_{\hat{\rho}-M\hat{\nu}}%
(\hat{y})$ and $t\in\lbrack0,\hat{\nu}+\lambda],$
\begin{align*}
e^{a(\hat{\nu}+t)}V_{\delta}(x(\hat{\nu}+t,y))+\int_{0}^{\hat{\nu}+t}%
W(x(\tau;y))d\tau &  \leq V(y)+\varepsilon e^{a\hat{\nu}}(\rho_{1}-\frac{\rho
}{2})(\rho-\rho_{2})\\
&  \leq V(y)+\varepsilon(e^{a\hat{\nu}}\rho_{1}-\frac{\rho}{2})(\rho-\rho
_{2}).
\end{align*}
Consequently, since that $\hat{\rho}-M\hat{\nu}\in(\frac{\rho}{2},\rho_{2})$
and $e^{a\hat{\nu}}\rho_{1}\in(\frac{\rho}{2},\rho_{2})$ it follows that
$[0,\lambda+\hat{\nu}]\subset E_{\varepsilon,\delta}$ and, so, the openness of
$E_{\varepsilon,\delta}$ follows.

In order to conclude the proof, let $y\in\overline{B}_{\frac{\rho}{2}}(\hat
{y})\cap G(\hat{y})$ be given. Then, for every $t\geq0$ we have that $t\in
\cap_{\varepsilon>0}E_{\varepsilon,\delta};$ that is for all $\varepsilon>0$
it holds
\[
e^{at}V_{\delta}(x(t; y))+\int_{0}^{t}W(x(\tau;y))d\tau\leq V(y)+\varepsilon
(\rho-\frac{\rho}{2})(\rho-\frac{\rho}{2})=V(y)+\varepsilon\frac{\rho^{2}}%
{4}.
\]
Hence,  letting $\varepsilon\rightarrow0$ it follows that
\[
e^{at}V_{\delta}(x(t; y))+\int_{0}^{t}W(x(\tau;y))d\tau\leq V(y),
\]
which as $\delta\rightarrow0$ yields (using the fact that $\lim_{\delta
\rightarrow0}V_{\delta}(x(t; y))=V(x(t; y))$  
\[
e^{at}V(x(t; y))+\int_{0}^{t}W(x(\tau;y))d\tau\leq V(y).
\]
Now, if $\bar{z}\in B_{\rho}(\hat{y})\cap\dom V,$ then similarly
as above, we can find $\rho_{\bar{z}}>0$ such that for every $z\in\overline
{B}_{\frac{\rho_{\bar{z}}}{2}}(\bar{z})\cap G(\bar{z})$ (where $G(\bar{z})$ is
defined as in (\ref{gyb})) we have that
\[
e^{at}V(x(t; z))+\int_{0}^{t}W(x(\tau;z))d\tau\leq V(z)\text{ \ for all }%
t\geq0.
\]
Thus, the main conclusion of the current theorem follows since that the last
inequality obviously holds when $\bar{z}\notin\dom V.$\smartqed 
\end{proof}

\begin{remark}
\label{corcon}\emph{The conclusion of Theorem \ref{heartlocalfinite}%
}\ \emph{also holds if, instead of }$V$\emph{ being weak continuous on
}$B_{\rho}(\bar{y}),$\emph{ we assume that either} $H$\emph{ is
finite-dimensional or} $V$ \emph{is convex}$.$
\end{remark}

\begin{proof}
The only difference with the proof of Theorem \ref{heartlocalfinite} arises in
showing (\ref{asd}).

(a) Assume that $H$ is finite-dimensional. Let us show that (\ref{asd}) holds.
Assuming the contrary, we find bounded sequencse $y_{k}\in B_{\frac{\rho}%
{2}+\frac{1}{k}}(\bar{y})\cap G(\bar{y})$ and $0<t_{k}\leq\tilde{\lambda}$
such that (\ref{hg}) holds. W.l.o.g. we may suppose that $t_{k}\rightarrow
\tilde{t}\leq\tilde{\lambda}$ and $y_{k}\rightharpoonup\tilde{y}\in
\overline{B}_{\frac{\rho}{2}}(\bar{y}).$ Furthermore, we have that
\[
V(\tilde{y})\leq\liminf_{k}V(y_{k})\leq\left\vert V(\bar{y})\right\vert +1,
\]
while (\ref{vlsc1}) guarantees that $V(\tilde{y})\geq V(\bar{y})-1.$ Hence, we
also have that $\tilde{y}\in\lbrack\left\vert V\right\vert \leq\left\vert
V(\bar{y})\right\vert +1].$ Now, recalling that $x(t_{k};y_{k})$ converges to
$x(\tilde{t},\tilde{y})$ in this case, it follows that
\begin{align*}
e^{a\tilde{t}}V_{\delta}(x(\tilde{t},\tilde{y}))&+\int_{0}^{\tilde{t}}%
W(x(\tau;\tilde{y}))d\tau\\
&  =e^{a\tilde{t}}V_{\delta}(\lim_{k}x(t_{k}%
,y_{k}))+\int_{0}^{\tilde{t}}W(\lim_{k}x(\tau;y_{k}))d\tau\\
&  =\lim_{k}\left[  e^{a\tilde{t}}V_{\delta}(x(t_{k};y_{k}))+\int_{0}%
^{\tilde{t}}W(x(\tau;y_{k}))d\tau\right] \\
&  \geq\liminf_{k}V(y_{k})+\varepsilon(\rho_{1}-\frac{\rho}{2})(\rho-\rho
_{2})\\
&  \geq V(\tilde{y})+\varepsilon(\rho_{1}-\frac{\rho}{2})(\rho-\rho
_{2})>V(\tilde{y}),
\end{align*}
which contradicts (\ref{vn}).

(b) Assume that $V$ is convex. We consider again the sequences of the proof of
Theorem \ref{heartlocalfinite}, $(y_{k})_{k\in\N}\subset B_{\frac{\rho}{2}+\frac
{1}{k}}(\bar{y})\cap G(\bar{y})\setminus\overline{B}_{\frac{\rho}{2}}(\bar
{y})$ and $(\tilde{y}_{k})_{k\in\N}\subset\overline{B}_{\frac{\rho}{2}}(\bar{y}),$
which both converge to $\tilde{y}\in\overline{B}_{\frac{\rho}{2}}(\bar{y})\cap
G(\bar{z}).$ Since that each $\tilde{y}_{k}\in\lbrack y_{k},\bar{y}],$ with
$k\geq1,$ we find $\beta_{k}\in\lbrack0,1]$ such that $\tilde{y}_{k}%
:=\beta_{k}y_{k}+(1-\beta_{k})\bar{y}$, this  yields
\[
V(\tilde{y}_{k})\leq\beta_{k}V(y_{k})+(1-\beta_{k})V(\bar{y}).
\]
We notice that $1\geq\beta_{k}\geq\frac{k\rho}{k\rho+2}$  since  by
construction,  $\tilde{y}_{k}$ is on the boundary of $\overline{B}_{\frac{\rho
}{2}}(\bar{y})$ and $y_{k}\in B_{\frac{\rho}{2}+\frac{1}{k}}(\bar{y}).$ Thus,
we may suppose that $\beta_{k}\rightarrow1.$ Consequently,  taking limits
in the inequality above,
\[
\liminf_{k}V(\tilde{y}_{k})\leq\lim_{k}\beta_{k}\liminf_{k}V(y_{k}%
)=\liminf_{k}V(y_{k}).
\]
Hence, as in (\ref{golf}),  using (\ref{lli}) we obtain that
\begin{align*}
\lim_{k} \Biggl[  e^{a\tilde{t}}V_{\delta}(x(t_{k};\tilde{y}_{k}))&+\int%
_{0}^{\tilde{t}}W(x(\tau;\tilde{y}_{k}))d\tau \Biggr]   \\
& =e^{a\tilde{t}}
\lim_{k}V_{\delta}(x(t_{k};\tilde{y}_{k}))+\int_{0}^{\tilde{t}}\lim
_{k}W(x(\tau;\tilde{y}_{k}))d\tau\\
& =e^{a\tilde{t}}\lim_{k}V_{\delta}(x(t_{k};y_{k}))+\int_{0}^{\tilde{t}}%
\lim_{k}W(x(\tau;y_{k}))d\tau\\
&=\lim_{k}\left[  e^{at_{k}}V_{\delta}(x(t_{k};y_{k}))+\int_{0}^{t_{k}%
}W(x(\tau;y_{k}))d\tau\right] \\
&  \geq\liminf_{k}V(y_{k})+\varepsilon(\rho_{1}-\frac{\rho}{2})(\rho-\rho
_{2})\\
& \geq\liminf_{k}V(\tilde{y}_{k})+\varepsilon(\rho_{1}-\frac{\rho}{2}%
)(\rho-\rho_{2}),
\end{align*}
which contradicts  (\ref{vn}).\smartqed 
\end{proof}

\bigskip

\begin{corollary}
\label{cn1}Assume that $\inte\left(  \operatorname*{co}%
\{\dom A\}\right)  \neq\emptyset.$ Let $V\in\mathcal{F}(H)$ be
convex, and let $W\in\mathcal{F}(H;\overline{{\mathbb{R}}}_{+})$ and
$a\in{\mathbb{R}}_{+}$ be given. Fix $\bar{y}\in\inte%
(\dom A)\cap\dom V,$ and let $\rho>0$ be such that
$B_{2\rho}(\bar{y})\subset\inte(\dom A).$ For all
$y\in B_{2\rho}(\bar{y})\cap\dom V$ we assume that
\[
\sup\limits_{\xi\in\partial_{P}V(y)}\inf\limits_{\upsilon \in Ay}\left\langle
\xi,f(y)-\upsilon \right\rangle +aV(y)+W(y)\leq0.
\]
Then, for all $y\in B_{\rho}(\bar{y})$ we have that
\[
e^{at}V(x(t; y))+\int_{0}^{t}W(x(\tau;y))d\tau\leq V(y)\text{ \ for all }%
t\geq0.
\]

\end{corollary}

\begin{proof}
According to Theorem \ref{heartlocalfinite} and Remark \ref{corcon}, it
suffices to show that the current assumption \textbf{i}mplies that, for every
given $y\in B_{2\rho}(\bar{y})\cap\dom V$ and $\xi\in
\partial_{\infty}V(y)=\mathrm{N}_{\dom V}(y)$ (if any), there
exists $\upsilon \in Ay$ such that
\begin{equation}
\langle\xi,f(y)-\upsilon \rangle\leq0. \label{clm}%
\end{equation}
To prove this fact, by the lsc of $V$ we let $\varepsilon>0$ be such that
\[
B_{\sqrt{\varepsilon}}(y)\subset\inte(\operatorname*{cl}%
(\dom A)),\text{ }V(B_{\sqrt{\varepsilon}}(y))\geq V(y)-1.
\]
Pick $y_{\varepsilon}\in\partial_{\varepsilon}V(y);$ this last set is not
empty since that $V\in\mathcal{F}(H)$ is a convex function. Then, from the
relationship $\mathrm{N}_{\dom V}(y)=(\partial_{\varepsilon
}V(y))_{\infty}$ (e.g. ), for every $k\in\mathbb{N}$ we have that
\[
y_{\varepsilon}+k\xi\in\partial_{\varepsilon}V(y).
\]
According to the Br\o ndsted-Rockafellar  Theorem, there are $y_{k}\in B_{\sqrt{\varepsilon}%
}(y)$ and $u_{k}\in B_{\sqrt{\varepsilon}}(\theta)$ such that
\[
y_{\varepsilon}+k\xi\in\partial_{\varepsilon}V(y_{k})+u_{k};
\]
that is, in particular, $y_{k}\in\dom V.$ Consequently, by the
current assumption we get that
\begin{align*}
k\langle\xi,f(y)-\pi_{Ay_{k}}(f(y_{k}))\rangle &  \leq\langle u_{k}%
-y_{\varepsilon},f(y_{k})-\pi_{Ay_{k}}(f(y_{k}))\rangle-aV(y_{k})-W(y_{k})\\
&  \leq\langle u_{k}-y_{\varepsilon},f(y_{k})-\pi_{Ay_{k}}(f(y_{k}%
))\rangle-aV(y)+a\\
&  +k\langle\xi,f(y)-f(y_{k})\rangle\\
&  \leq\langle u_{k}-y_{\varepsilon},f(y_{k})-\pi_{Ay_{k}}(f(y_{k}%
))\rangle-aV(y)+a+kL_{f}\sqrt{\varepsilon}\left\Vert \xi\right\Vert .
\end{align*}
Since that $\langle u_{k}-y_{\varepsilon},f(y_{k})-\pi_{Ay_{k}}(f(y_{k}%
))\rangle$ is bounded independently of $k$, for $k_{\varepsilon}\geq1$ big
enough we get that
\[
\langle\xi,f(y)-\pi_{Ay_{k_{\varepsilon}}}(f(y_{k_{\varepsilon}}))\rangle
\leq\sqrt{\varepsilon}+L_{f}\sqrt{\varepsilon}\left\Vert \xi\right\Vert .
\]
Moreover, as $y_{k_{\varepsilon}}\in B_{\sqrt{\varepsilon}}(y)$ and
$\zeta_{\varepsilon}:=\pi_{Ay_{k_{\varepsilon}}}(f(y_{k_{\varepsilon}}))$
($\in Ay_{k_{\varepsilon}}$) is bounded independently of $k_{\varepsilon},$ we
may suppose as $\varepsilon\rightarrow0$\ that $(\zeta_{\varepsilon})$ weakly
converges to some $\upsilon \in Ay.$ Thus, taking limits in the last inequality
above we get that $\langle\xi,f(y)-\upsilon \rangle\leq0;$ that is (\ref{clm}) follows.\smartqed 
\end{proof}

\begin{corollary}
\label{cn2}Assume that $\dim H<\infty.$ Let $V\in\mathcal{F}(H)$,
$W\in\mathcal{F}(H;\overline{{\mathbb{R}}}_{+}),$ and $a\in{\mathbb{R}}_{+}$
be given. Fix $\bar{y}\in\inte(\dom A),$ and let
$\rho>0$ be such that $B_{2\rho}(\bar{y})\subset\inte%
(\dom A).$ For all $y\in B_{2\rho}(\bar{y})\cap
\dom V$ we assume that
\[
\sup\limits_{\xi\in\partial_{P}V(y)}\inf\limits_{\upsilon \in Ay}\left\langle
\xi,f(y)-\upsilon \right\rangle +aV(y)+W(y)\leq0.
\]
Then, for all $y\in B_{\rho}(\bar{y})$ we have that
\[
e^{at}V(x(t; y))+\int_{0}^{t}W(x(\tau;y))d\tau\leq V(y)\text{ \ for all }%
t\geq0.
\]

\end{corollary}

\begin{proof}
As in the proof of Corollary \ref{cn1}, given $y\in B_{2\rho}(\bar{y}%
)\cap\dom V$ and $\xi\in\partial_{\infty}V(y)$ (if any)$,$ we
only to find some $\upsilon \in Ay$ such that
\[
\langle\xi,f(y)-\upsilon \rangle\leq0.
\]
Fix $\varepsilon>0.$ By definition, we let $\xi_{k}\in\partial_{P}V(y_{k})$
and $\alpha_{k}\downarrow0$ such that $y_{k}\rightarrow y$, $V(y_{k}%
)\rightarrow V(y),$ and $\alpha_{k}\xi_{k}\rightharpoonup\xi.$ Then, by the
current assumption, for each $k$ there exists $y_{k}^{\ast}\in Ay_{k}$ such
that
\[
\left\langle \xi_{k},f(y_{k})-y_{k}^{\ast}\right\rangle +aV(y_{k}%
)+W(y_{k})\leq\varepsilon.
\]
Because $\dim H<\infty$ and $y_{k},y_{k}^{\ast}$ are bounded, we may suppose
that $y_{k}^{\ast}$ converges to some $\upsilon \in Ay.$ Thus, multiplying the
equation above by $\alpha_{k}$ and next passing to the limit as $\varepsilon
\rightarrow0$  and finally  invoking the lsc of $V$ and the Lipschitz continuity
of $f$,  we obtain  that
\[
\left\langle \xi,f(y)-\upsilon \right\rangle \leq\lim_{k}\left\langle
\alpha_{k}\xi_{k},f(y_{k})-y_{k}^{\ast}\right\rangle +a\liminf_{k}\alpha
_{k}V(y_{k})\leq\lim_{k}\alpha_{k}\varepsilon=0.
\]
The conclusion follows.\smartqed 
\end{proof}

\begin{remark}
\emph{Let }$V\in\mathcal{F}(H),$\emph{ and let} $W\in\mathcal{F}%
(H;\overline{{\mathbb{R}}}_{+})$ \emph{be Lipschitz continuous on}
$\operatorname*{cl}(\dom A).$ \emph{Define the the mapping}
$\tilde{f}:H\times\mathbb{R}^{2}\mathbb{\rightarrow}H\times\mathbb{R}^{2}$
\emph{and the operator} $\widetilde{A}:H\times\mathbb{R}^{2}%
\mathbb{\rightrightarrows}H\times\mathbb{R}^{2}$ \emph{respectively as}
\[
\tilde{f}(y,\alpha,\gamma):=\left(
\begin{array}
[c]{c}%
f(y)\\
W(y)\\
0
\end{array}
\right)  \text{ \emph{and} }\widetilde{A}(y,\alpha,\gamma):=\left(
\begin{array}
[c]{c}%
Ay\\
0\\
0
\end{array}
\right)  ,
\]
\emph{and denote }$\widetilde{V}:H\times\mathbb{R\rightarrow}\overline
{\mathbb{R}}$ \emph{and} $\widehat{V}:H\times\mathbb{R}^{2}\mathbb{\rightarrow
}\overline{\mathbb{R}}$\emph{\ the functions given respectively as}
\[
\widetilde{V}(y,\alpha):=V(y)+\alpha\text{ \emph{and} }\widehat{V}%
(y,\alpha,\gamma):=\mathrm{I}_{\operatorname*{epi}\widetilde{V}\cap
\operatorname*{cl}(\dom \widetilde{A})}(y,\alpha,\gamma).
\]
\emph{Consider the differential inclusion}
\begin{equation}
\dot{z}(t; (y,\alpha,\gamma))\in\tilde{f}(z(t; (y,\alpha,\gamma)))-\widetilde{A}%
z(t; (y,\alpha,\gamma)),\text{ \ \ }z(0,(y,\alpha,\gamma))=(y,\alpha,\gamma),
\label{diaug}%
\end{equation}
\emph{the solution of which is the function }$z(t; (y,\alpha,\gamma
)):[0,\infty)\rightarrow H\times\mathbb{R}^{2}$ \emph{given by}
\[
z(t; (y,\alpha,\gamma))=\left(
\begin{array}
[c]{c}%
x(t; y)\\
\int_{0}^{t}W(x(\tau;y))d\tau+\alpha\\
\gamma
\end{array}
\right)  .
\]
\emph{Then, }$(V,W)$ \emph{is a Lyapunov pair for} \emph{(\ref{id}) if and
only if the function} $\mathrm{I}_{\operatorname*{epi}\widetilde{V}%
\cap\operatorname*{cl}(\dom \widetilde{A})}$ \emph{is a Lyapunov
function for this new differential inclusion (\ref{diaug}).}
\end{remark}


\section{Characterizations of finite-dimensional nonsmooth Lyapunov pairs}

This section is devoted to the finite-dimensional setting. Assuming that $\dim
H<\infty,$ we give multiple primal and dual characterizations for nonsmooth
$a$-Lyapunov pairs for the differential inclusion (\ref{id}), with respect
to the set $\operatorname*{rint}(\operatorname*{cl}(\dom A))$.
Naturally, these conditions turn out to be sufficient   for nonsmooth
$a$-Lyapunov functions  with respect to every given set $D\subset\operatorname*{cl}%
(\dom A)$ verifying condition (\ref{regular}).

Further, in this setting, the dual characterization does not depend on the
choice of the subdifferential operator which can be either the proximal, the
Fr\'echet, the Limiting (which coincides with the viscosity subdifferential (see
Borwein  \cite{BorweinZhuBook}), or, more generally, every subdifferential operator $\partial
V:H\rightrightarrows H$ satisfying
\begin{equation}
\partial_{P}V\subset\partial V\subset\partial_{L}V, \label{subdifoper}%
\end{equation}
where $V\in\mathcal{F}(H)$ is the first part of Lyapunov's condidate pairs.


\begin{proposition}
\label{heartlocalfinitedimoension}Assume that $\dim H<\infty.$ Let
$V\in\mathcal{F}(H)$, $W\in\mathcal{F}(H;\overline{{\mathbb{R}}}_{+}),$ and
$a\in{\mathbb{R}}_{+}$ be given, and let $\partial$ be as in
\emph{(\ref{subdifoper})}. Fix $\bar{y}\in\operatorname*{rint}%
(\operatorname*{cl}(\dom A))$ and let $\rho>0$ be such that
$B_{2\rho}(\bar{y})\cap\operatorname*{aff}(\operatorname*{cl}%
(\dom A))\subset\dom A.$ Then, the following
assertions \emph{(i)--(v)} are equivalent\emph{:}\newline

\emph{(i)} for every $y\in\dom A\cap\dom V\cap
B_{\rho}(\bar{y})$
\[
e^{at}V(x(t; y))+\int_{0}^{t}W(x(\tau; y))d\tau\leq V(y)\text{ \ \ for all
}t\geq0;
\]

\emph{(ii) }for every $y\in\dom A\cap\dom V\cap
B_{\rho}(\bar{y})$
\[
\sup_{\xi\in\partial_{P}V(y)}\left\langle \xi,f(y)-\pi_{Ay}(f(y))\right\rangle
+aV(y)+W(y)\leq0;
\]

\emph{(iii) }for every $y\in\dom A\cap\dom V\cap
B_{\rho}(\bar{y})$
\[
\sup_{\xi\in\partial V(y)}\inf_{\upsilon \in Ay}\left\langle \xi,f(y)-y^{\ast
}\right\rangle +aV(y)+W(y)\leq0;
\]

\emph{(iv) }for every $y\in\dom A\cap\dom V\cap
B_{\rho}(\bar{y})$
\[
V^{\prime}(y;f(y)-\pi_{Ay}(f(y)))+aV(y)+W(y)\leq0;
\]

\emph{(v) }for every $y\in\dom A\cap\dom V\cap
B_{\rho}(\bar{y})$
\[
\inf_{\upsilon \in Ay}V^{\prime}(y;f(y)-\upsilon)+aV(y)+W(y)\leq0.
\]
If $V$ is nonnegative, each one of the statements above is equivalent to

\emph{(vi) }for every $y\in\dom A\cap\dom V\cap
B_{\rho}(\bar{y})$
\[
V(x(t; y))+a\int_{0}^{t}V(x(\tau; y))d\tau+\int_{0}^{t}W(x(\tau; y))d\tau\leq
V(y)\text{ \ \ for all }t\geq0.
\]

\end{proposition}

\begin{proof}
\textbf{(iii with }$\partial\equiv\partial_{P}$\textbf{)}$\implies
$\textbf{(i)}: Let $H_{0}:=\operatorname{lin}(\operatorname*{cl}%
(\dom A))$ denote the linear hull of $\dom A;$ we
may suppose that $\theta\in\dom A.$ Let $A_{0}:H_{0}%
\rightrightarrows H_{0}$ be the operator given by
\begin{equation}
A_{0}y=Ay\cap H_{0}, \label{az}%
\end{equation}
and define the Lipschitz continuous mapping $f_{0}:H_{0}\rightarrow H_{0}$ as
\begin{equation}
f_{0}(y)=\pi_{H_{0}}(f(y)), \label{fzz}%
\end{equation}
where $\pi_{H_{0}}$ denotes the orthogonal projection onto $H_{0}.$ According to  the Minty
Theorem,  it follows that  $A_{0}$ is also a maximally monotone operator. Further,
 for every $y\in\dom A$  we have  $Ay+\mathrm{N}%
_{\operatorname*{cl}(\dom A)}(y)=Ay,$  and therefore 
$Ay+H_{0}^{\perp}=Ay$.  Hence,
\begin{equation}
Ay=\left(  Ay\cap H_{0}\right)  +H_{0}^{\perp}=A_{0}y+H_{0}^{\perp}.
\label{boon}%
\end{equation}
From this inequality we deduce that $\dom A_{0}%
=\dom A$ and, so,
\[
\operatorname*{rint}(\operatorname*{cl}(\dom %
A))=\inte (\operatorname*{cl}(\dom A_{0}%
))=\inte (\dom A_{0});
\]
for the last equality see, e.g., \cite[Remark 2.1- Page 33]{BrezisBook}. Further,
since for $y\in\operatorname*{cl}(\dom A)$ we have that
\[
f_{0}(y)-A_{0}y\subset f(y)-A_{0}y+H_{0}^{\perp}=f(y)-Ay,
\]
 from which it   follows that $x\cdot; y)$ is the unique solution of the the differential
inclusion
\[
\dot{x}(t; y)\in f_{0}(x(t; y))-A_{0}x(t; y),\text{ }x(0,y)=y.
\]

Next,  we are going to show the assumption of Corollary \ref{cn2} (which is the
same as Conditions (i) of Theorem \ref{heartlocalfinite}) holds with respect
to the pair $(A_{0},f_{0}).$ Fix $y\in\dom A\cap
\dom V\cap B_{\rho}(\bar{y})$ and $\xi\in\partial V(y)$ (if
any). For fixed $\varepsilon>0,$ by assumption  take  $\upsilon \in Ay$  in such a way 
that
\[
\left\langle \xi,f(y)-\upsilon \right\rangle +aV(y)+W(y)\leq\varepsilon.
\]
Since $f(y)\in f_{0}(y)+H_{0}^{\perp}$ and $\upsilon +H_{0}^{\perp}\in
Ay+H_{0}^{\perp}=A_{0}y,$ we have
\begin{equation}
\inf_{\upsilon \in A_{0}y}\left\langle \xi,f_{0}(y)-\upsilon \right\rangle
\leq\inf_{\upsilon \in Ay}\left\langle \xi,f(y)-\upsilon \right\rangle
\leq\varepsilon-aV(y)-W(y), \label{tb}%
\end{equation}
and  the assumption of Corollary \ref{cn2} follows as $\varepsilon
\rightarrow0.$

\textbf{(i)}$\implies$\textbf{(iv)}: Fix $y\in\dom %
A\cap\dom V\cap B_{\rho}(\bar{y}).$ Then, as shown in the
paragraph above, the solution $x(t; y)$ of (\ref{id}) is also the unique
strong solution of the equation
\[
\dot{x}(t; y)\in f_{0}(x(t; y))-A_{0}x(t; y),\text{ \ }x(0; y)=y\in
\operatorname*{cl}\left(  \dom A\right) ,
\]
where $A_{0}$ and $f_{0}$ are defined in (\ref{az}) and (\ref{fzz}),
respectively. Let $(t_n)_{n\in\N}\subset(0,T)$ be such that $t_{n}\rightarrow
0^{+}$ and set
\[
w_{n}:=\frac{x(t_{n}; y)-y}{t_{n}}.
\]
Because $x(\cdot; y)$ is derivable from the right at $0$ ($y\in
\dom A$) and 
$$\displaystyle \frac{d^{+}x(\cdot; y)}{dt}(0)=(f(y)-Ay)^{\circ
}=f(y)-\pi_{Ay}(f(y)), $$
we infer that
\[
w_{n}\rightarrow f(y)-\pi_{Ay}(f(y)).
\]
Therefore, using the current assumption (i),
\begin{align*}
\frac{V(y+t_{n}w_{n})-V(y)}{t_{n}}\\&=\frac{V(x(t_{n},y))-V(y)}{t_{n}}\\&\leq
\frac{e^{-at_{n}}(1-e^{at_{n}})}{t_{n}}V(y)-\frac{e^{-at_{n}}}{t_{n}}\int%
_{0}^{t_{n}}W(x(s; y))ds,
\end{align*}
and  taking limits  yields
\begin{align*}
V^{\prime}(y;f(y)-\pi_{Ay}(f(y)))  &  \leq\liminf_{n}\frac{e^{-at_{n}%
}(1-e^{at_{n}})}{t_{n}}V(y)-\frac{e^{-at_{n}}}{t_{n}}\int_{0}^{t_{n}%
}W(x(s; y))ds\\
&  =-aV(y)-W(y);
\end{align*}
this proves (iv).

(iv)$\implies$(v) is trivial.

(v) $\implies$ (iii).  Use $\partial\equiv\partial_{L}.%
$\textbf{)}: Take $y\in\dom A\cap\dom V\cap
B_{\rho}(\bar{y}).\ $For fixed $\varepsilon>0,$ by (v) we let $\upsilon \in Ay$
be such that
\[
V^{\prime}(y;f(y)-\upsilon )\leq\varepsilon-aV(y)-W(y);
\]
that is
\[
(f(y)-\upsilon ,\varepsilon-aV(y)-W(y))\in\operatorname*{epi}V^{\prime}%
(y,\cdot)=\mathrm{T}_{\operatorname*{epi}V}(y,V(y))\subset\left[
\mathrm{N}_{\operatorname*{epi}V}^{p}(y,V(y))\right]  ^{\circ}.
\]
If $\xi\in\partial_{P}V(y)$, since that $(\xi,-1)\in\mathrm{N}%
_{\operatorname*{epi}V}^{p}(y,V(y))$ the last  above inequality  leads us to
\begin{align*}
\left\langle \xi,f(y)-\upsilon \right\rangle  &  \leq\left\langle
(\xi,-1),(f(y)-\upsilon ,\varepsilon-aV(y)-W(y))\right\rangle +\varepsilon
-aV(y)-W(y)\\
&  \leq\varepsilon-aV(y)-W(y)
\end{align*}
so that (ii) follows when $\varepsilon\rightarrow0.$

If $\xi\in\partial_{L}V(y)$, then there are sequences $y_{n}\rightarrow
y,\xi_{n}\rightarrow\xi$ such that $V(\xi_{n})\rightarrow V(\xi)$ and $\xi
_{n}\in V(y_{n})$ for every integer  $n$ sufficiently large. As just shown above, given an $\varepsilon
>0,$ for each $n$ there exists $y_{n}^{\ast}\in Ay_{n}$ such that
\[
\left\langle \xi_{n},f(y_{n})-y_{n}^{\ast}\right\rangle \leq\varepsilon
-aV(y_{n})-W(y_{n}).
\]
Because $(y_{n})_{n}\subset B_{\rho}(\bar{y})\subset\inte %
(\dom A_{0})\subset H_{0}$ (the ball $B_{\rho}(\bar{y})$ is with
respect to $H_{0}$), then we may suppose that $y_{n}^{\ast}\rightarrow
\upsilon \in Ay.$ Thus,  passing to the limit in the above inequality, and
taking into account the lsc of $V$ and the continuity of $W,$
\[
\left\langle \xi,f(y)-\upsilon \right\rangle \leq\varepsilon-aV(y)-W(y).
\]
showing that (iii) holds with $\partial\equiv\partial_{L}.$

At this point we have proved that (i)$\Longleftrightarrow$(iii with
$\partial\equiv\partial_{L}$)$\Longleftrightarrow$(iv)$\Longleftrightarrow
$(v). To see that (ii) is also equivalent to the other statements we observe
that, from one hand, (ii)$\implies$(iii) holds obviously. On the other hand,
the implication (iv)$\implies$(ii) follows in a similar way as in the proof of
the statement (v)$\implies$(iii). This finishes the proof of the equivalences
of (i) through (v).

Finally, if $V$ is nonnegative, (vi) is nothing else but (i) with $a$ and $W$
replaced by $\theta$ and $aV+W,$ respectively. Thus, (vi) is equivalent to (iii).\smartqed 
  \end{proof}
The following Theorem, which is an immediate consequence of Proposition
\ref{subdifoper}\emph{,} establishes  primal and dual
characterizations of Lyapunov pairs for (\ref{id}) with respect to
$\operatorname*{rint}(\operatorname*{cl}(\dom A))$. Sufficient
conditions for Lyapunov pairs with respect to other sets are then deduced
under (\ref{regular}).

\begin{theorem}
\label{lyapfinite}Assume that $\dim H<\infty.$ Let $V\in\mathcal{F}%
(H),W\in\mathcal{F}(H;\overline{\mathbb{R}}_{+}),$ and $a\in\mathbb{R}_{+}$ be
given, and let $\partial$ be as in \emph{(\ref{subdifoper})}. Then, $(V,W)$
forms an $a$-Lyapunov pair for \emph{(\ref{id}),} with respect to
$\operatorname*{rint}(\operatorname*{cl}(\dom A)),$ if and only
if one of the following assertions holds\emph{:}\newline

\emph{(i) }for all $y\in\operatorname*{rint}(\operatorname*{cl}%
(\dom A))\cap\dom V$
\[
\sup_{\xi\in\partial_{P}V(y)}\left\langle \xi,f(y)-\pi_{Ay}(f(y))\right\rangle
+aV(y)+W(y)\leq0;
\]

\quad\newline

\emph{(ii) }for all $y\in\operatorname*{rint}(\operatorname*{cl}%
(\dom A))\cap\dom V$
\[
\sup_{\xi\in\partial V(y)}\inf_{\upsilon \in Ay}\left\langle \xi,f(y)-y^{\ast
}\right\rangle +aV(y)+W(y)\leq0;
\]

\quad\newline

\emph{(iii) }for all $y\in\operatorname*{rint}(\operatorname*{cl}%
(\dom A))\cap\dom V$%
\[
V^{\prime}(y;f(y)-\pi_{Ay}(f(y)))+aV(y)+W(y)\leq0;
\]

\quad\newline

\emph{(iv) }for all $y\in\operatorname*{rint}(\operatorname*{cl}%
(\dom A))\cap\dom V$
\[
\inf_{\upsilon \in Ay}V^{\prime}(y;f(y)-\upsilon )+aV(y)+W(y)\leq0.
\]
Consequently, if $V$ satisfies \emph{(\ref{regular})} for a given set
$D\subset\operatorname*{cl}(\dom A),$ then any of the conditions
\emph{(i)-(iv)} above implies that $(V,W)$ is an $a$-Lyapunov pair for
\emph{(\ref{id})} with respect to $D.$
\end{theorem}

In contrast to the (analytic) Definition \ref{deflyap}, Lyapunov stability can
also be approached from a geometrical point of view using the concept of invariance:


\begin{definition}
\label{invariance} Let be given a set $D\subset\operatorname*{cl}%
(\dom A).$ A non-empty closed set $S\subset H$ is said invariant
for \emph{(\ref{id})} with respect to $D$ if for all $y\in S\cap D$ one has
that
\[
x(t; y)\in S\text{ \ for all }t\geq0.
\]

\end{definition}

This fact, which was already mentioned in the \textbf{infinite-dimensional} setting in
Corollary \ref{cn1}, is explicitly characterized here in the
\textbf{finite-dimensional }setting. This characterization is also valid in the
infinite-dimensional setting provided that $S\cap\operatorname*{cl}%
(\dom A)$ is a convex set,  according  to Remark \ref{corcon}
and Corollary \ref{cn1}.


\begin{corollary}
Assume that $\dim H<\infty.$ A closed set $\emptyset\neq S\subset H$ is
invariant for \emph{(\ref{id}),} with respect to $\operatorname*{rint}%
(\operatorname*{cl}(\dom A)),$ if and only if one of the
following assertions are  satisfied\emph{:}\newline

\emph{(i) }for all $y\in\operatorname*{rint}(\operatorname*{cl}%
(\dom A))\cap S$
\[
\sup_{\xi\in\mathrm{N}_{S\cap\operatorname*{cl}(\dom A)}^{P}%
(y)}\left\langle \xi,f(y)-\pi_{Ay}(f(y))\right\rangle \leq0;
\]

\emph{(ii) }for all $y\in\operatorname*{rint}(\operatorname*{cl}%
(\dom A))\cap S$
\[
\sup_{\xi\in\mathrm{N}_{S\cap\operatorname*{cl}(\dom A)}^{P}%
(y)}\inf_{\upsilon \in Ay}\left\langle \xi,f(y)-\upsilon \right\rangle \leq0;
\]

\emph{(iii) }for all $y\in\operatorname*{rint}(\operatorname*{cl}%
(\dom A))\cap S$
\[
f(y)-\pi_{Ay}(f(y))\in T_{S\cap\operatorname*{cl}(\dom A)}(y);
\]

\emph{(iv) }for all $y\in\operatorname*{rint}(\operatorname*{cl}%
(\dom A))\cap S$
\[
\left[  f(y)-Ay\right]  \cap T_{S\cap\operatorname*{cl}(\dom %
A)}(y)\neq\emptyset;
\]

\emph{(v) }for all $y\in\operatorname*{rint}(\operatorname*{cl}%
(\dom A))\cap S$
\[
\left[  f(y)-Ay\right]  \cap\overline{\operatorname*{co}}\left[
T_{S\cap\operatorname*{cl}(\dom A)}(y)\right]  \neq\emptyset.
\]
Consequently, $S$ is invariant for \emph{(\ref{id})} with respect to a given
set $D\subset\operatorname*{cl}(\dom A)$ if
\[
S\cap D\subset\operatorname*{cl}(S\cap\operatorname*{rint}(\operatorname*{cl}%
(\dom A))).
\]

\end{corollary}

\begin{proof}
It is an immediate fact that, with respect to $\operatorname*{rint}%
(\operatorname*{cl}(\dom A)),$ $S$ is invariant if and only if
$\mathrm{I}_{S\cap\operatorname*{cl}(\dom A)}$ is a Lyapunov
function. Then, the current assertions  (i) and (ii) come from statements (i) and (ii) of
Proposition \ref{heartlocalfinitedimoension}, respectively. Similarly, always
with respect to $\operatorname*{rint}(\operatorname*{cl}(\dom %
A)),$ $S$ is invariant if and only $d(\cdot,S\cap\operatorname*{cl}%
(\dom A))$ is a Lyapunov function. Thus,    by virtue of the
relationship
\[
T_{S\cap\operatorname*{cl}(\dom A)}(y)=\{w\in H\mid d^{\prime
}(\cdot,S\cap\operatorname*{cl}(\dom A)(w)=0\},
\]
the current  assertions (iii) and (iv) follow from statements (iii) and (iv) of
Proposition \ref{heartlocalfinitedimoension}, respectively. This shows that
(i)$\Longleftrightarrow$(ii)$\Longleftrightarrow$(iii)$\Longleftrightarrow$(iv).

It remains to show that (v) is equivalent to the other statements. We
obviously have that (iv)$\implies$(v) and so (i)$\implies$(v). To prove the
reverse implication it suffices to show that (v)$\implies$(ii). Indeed, fix
$y\in S\cap\dom A$ and $\xi\in\mathrm{N}_{S\cap
\operatorname*{cl}(\dom A)}^{P}.$ Then, by (v) there exists
$\upsilon \in$ $Ay$ such that
\[
f(y)-\upsilon \in\overline{\operatorname*{co}}\left[  T_{S\cap
\operatorname*{cl}(\dom A)}(y)\right]  \subset\left[
\mathrm{N}_{S\cap\operatorname*{cl}(\dom A)}^{P}\right]
^{\circ}.
\]
Therefore, $\left\langle \xi,f(y)-\upsilon \right\rangle \leq0;$ that is (ii) follows.\smartqed 
  \end{proof}
\vskip 2mm
The following corollary follows from Theorem \ref{lyapfinite}.

\begin{corollary}
\label{convcont}Assume that $\dim H<\infty.$ Let $V\in\mathcal{F}(H)$,
$W\in\mathcal{F}(H,\mathbb{R}_{+}),$ and $a\in{\mathbb{R}}_{+}$ be given, and
let $\partial$ be as in \emph{(\ref{subdifoper})}. Then, the following
statements are equivalent provided that $V$ is continuous relative to
$\operatorname*{cl}(\dom A)$\emph{:}

\emph{(i)} $(V,W)$ is an $a$-Lyapunov pair for \emph{(\ref{id})} with
respect to $\operatorname*{cl}(\dom A);$

\emph{(ii)} $(V,W)$ is an $a$-Lyapunov pair for \emph{(\ref{id})} with
respect to $\operatorname*{rint}(\operatorname*{cl}(\dom A));$

\emph{(iii)} for every $y\in\operatorname*{rint}(\operatorname*{cl}%
(\dom A))\cap\dom V$
\[
\sup_{\xi\in\partial V(y)}\langle\xi,f(y)-\pi_{Ay}(f(y))\rangle+aV(y)+W(y)\leq
0;
\]

\emph{(iv) }for all $y\in\operatorname*{rint}(\operatorname*{cl}%
(\dom A))\cap\dom V$
\[
\inf_{\upsilon \in Ay}V^{\prime}(y;f(y)-\upsilon )+aV(y)+W(y)\leq0.
\]

\end{corollary}

The characterization of G\^ateaux differentiable Lyapunov functions is a special
case of the following corollary.


\begin{corollary}
\label{SmoothLyap}Assume that $\dim H<\infty.$ Let $V\in\mathcal{F}(H)$,
$W\in\mathcal{F}(H,\mathbb{R}_{+}),$ and $a\in{\mathbb{R}}_{+}$ be given. If
$V$ is G\^ateaux differentiable, then the following statements are
equivalent\emph{:}

\emph{(i)} $(V,W)$ is an $a$-Lyapunov pair for \emph{(\ref{id})} with
respect to $\operatorname*{cl}(\dom A);$

\emph{(ii)} $(V,W)$ is an $a$-Lyapunov pair for \emph{(\ref{id})} with
respect to $\operatorname*{rint}(\operatorname*{cl}(\dom A));$

\emph{(iii)} for every $y\in\operatorname*{rint}(\operatorname*{cl}%
(\dom A))\cap\dom V$
\[
V_{G}^{\prime}(y)(f(y)-\pi_{Ay}(f(y)))+aV(y)+W(y)\leq0;
\]

\emph{(iv) }for all $y\in\operatorname*{rint}(\operatorname*{cl}%
(\dom A))\cap\dom V$
\[
\inf_{\upsilon \in Ay}V_{G}^{\prime}(y)(f(y)-\upsilon )+aV(y)+W(y)\leq0.
\]

\end{corollary}

In order to fix ideas,  let us discuss the simple case when $A\equiv0$ so that our
inclusion (\ref{id}) becomes an ordinary differential equation which reads:
for every $y\in H$ there exists a unique $x\cdot; y)\in C^{1}(0,\infty;H)$
such that $x(0,y)=y$ and
\begin{equation}
\dot{x}(t;y)=f(x(t; y))\text{ for all }t\geq0. \label{ide}%
\end{equation}
In this case, Theorem \ref{heartlocal} gives in a simplified form the
characterization of the associated $a$-Lyapunov pairs.


\begin{corollary}
\label{ODELyap}Assume that $\dim H<\infty.$ Let be given $V\in\mathcal{F}(H),$
$W\in\mathcal{F}(H;\overline{{\mathbb{R}}}_{+}),$ and $a\in{\mathbb{R}}_{+}$.
The following stataments are equivalent\emph{:}

\emph{(i)} $(V,W)$ is an $a$-Lyapunov pair for \emph{(\ref{ide})} (with
respect to $H)$\emph{;}

\emph{(ii)} for every $y\in\dom V$
\[
V^{\prime}(y;f(y))+aV(y)+W(y)\leq0;
\]

\emph{(iii)} for all $y\in\dom V$
\[
\sup_{\xi\in\partial V(y)}\langle\xi,f(y)\rangle+aV(y)+W(y)\leq0,
\]

\quad\newline where $\partial V$ stands for any subdifferential operator
verifying $\partial_{P}V\subset\partial V\subset\partial_{C}V.$
\end{corollary}

\begin{proof}
By Theorem \ref{lyapfinite} the conclusion  holds for all the subdifferentials
$\partial V$ such that $\partial_{P}V\subset\partial V\subset\partial_{L}V.$
To show that (iii) is also a characterization when $\partial\equiv\partial
_{C}$ it suffices, in view of the relationship $\partial_{L}\subset
\partial_{C},$ to show that (iii with $\partial\equiv\partial_{L}$) implies
(iii with $\partial\equiv\partial_{C}$). Indeed, fix $y\in\dom %
V$ so that
\[
\sup_{\xi\in\partial_{\infty}V(y)}\langle\xi,f(y)\rangle\leq0.
\]
So, according to \cite {RockBook2002}, (iii with $\partial\equiv\partial_{C}$) follows
since that
\[
\sup_{\xi\in\partial_{C}V(y)}\langle\xi,f(y)\rangle+aV(y)+W(y)=\sup_{\xi
\in\overline{\operatorname*{co}}\{\partial_{L}V(y)+\partial_{\infty}%
V(y)\}}\langle\xi,f(y)\rangle+aV(y)+W(y)\leq0.
\]
\smartqed 
\end{proof}
\begin{acknowledgement}
 Research partially supported by the project BQR R\'eseaux  B005/R1107005
from the university of Limoges, by the Australian Research Council under grant DP-110102011 and by the
ECOS-SUD C10E08 project.
\end{acknowledgement}
\bibliographystyle{plain}
\bibliography{myrefsAHTII}

\begin{thebibliography}{10}

\bibitem{Adlyetal04}
S.~Adly and D.~Goeleven.
\newblock A stability theory for second-order nonsmooth dynamical systems with
  application to friction problems.
\newblock {\em J. Math. Pures Appl. (9)}, 83(1):17--51, 2004.

\bibitem{AHT2012}
S.~Adly, A.~Hantoute, and M.~Th{\'e}ra.
\newblock Nonsmooth {L}yapunov pairs for infinite-dimensional first-order
  differential inclusions.
\newblock {\em Nonlinear Anal.}, 75(3):985--1008, 2012.

\bibitem{BC99}
A.~Bacciotti and F.~Ceragioli.
\newblock Stability and stabilization of discontinuous systems and nonsmooth
  {L}yapunov functions.
\newblock {\em ESAIM Control Optim. Calc. Var.}, 4:361--376 (electronic), 1999.

\bibitem{BC2011}
H.~Bauschke and P.~L. Combettes.
\newblock {\em Convex analysis and monotone operator theory in {H}ilbert
  spaces}.
\newblock CMS Books in Mathematics/Ouvrages de Math\'ematiques de la SMC.
  Springer, New York, 2011.

\bibitem{Borwein10}
J.~M. Borwein.
\newblock Fifty years of maximal monotonicity.
\newblock {\em Optim. Lett.}, 4(4):473--490, 2010.

\bibitem{BorYao2012}
J.~M. Borwein and L.~Yao.
\newblock Recent progress on monotone operator theory.
\newblock {\em arXiv}, (1210.3401), 2012.

\bibitem{BorweinZhuBook}
J.~M. Borwein and Q.~J. Zhu.
\newblock {\em Techniques of variational analysis}.
\newblock CMS Books in Mathematics/Ouvrages de Math\'ematiques de la SMC, 20.
  Springer-Verlag, New York, 2005.

\bibitem{BrezisBook}
H.~Br{\'e}zis.
\newblock {\em Op\'erateurs maximaux monotones et semi-groupes de contractions
  dans les espaces de {H}ilbert}.
\newblock North-Holland Publishing Co., Amsterdam, 1973.
\newblock North-Holland Mathematics Studies, No. 5. Notas de Matem{\'a}tica
  (50).

\bibitem{cps1}
M.~K. Camlibel, J.-S. Pang, and J.~Shen.
\newblock Conewise linear systems: non-{Z}enoness and observability.
\newblock {\em SIAM J. Control Optim.}, 45(5):1769--1800, 2006.

\bibitem{cps2}
M.~K. Camlibel, J.-S. Pang, and J.~Shen.
\newblock Lyapunov stability of complementarity and extended systems.
\newblock {\em SIAM J. Optim.}, 17(4):1056--1101, 2006.

\bibitem{carja-motreanu1}
O.~C{\^a}rj{\u{a}} and D.~Motreanu.
\newblock Flow-invariance and {L}yapunov pairs.
\newblock {\em Dyn. Contin. Discrete Impuls. Syst. Ser. A Math. Anal.},
  13B(suppl.):185--198, 2006.

\bibitem{Carjaetal09-flow}
O.~C{\^a}rj{\u{a}} and D.~Motreanu.
\newblock Characterization of {L}yapunov pairs in the nonlinear case and
  applications.
\newblock {\em Nonlinear Anal.}, 70(1):352--363, 2009.

\bibitem{ClarkeBook}
F.~H. Clarke.
\newblock {\em Optimization and nonsmooth analysis}, volume~5 of {\em Classics
  in Applied Mathematics}.
\newblock Society for Industrial and Applied Mathematics (SIAM), Philadelphia,
  PA, second edition, 1990.

\bibitem{Clarke2009}
F.~H. Clarke.
\newblock Lyapunov functions and feedback in nonlinear control.
\newblock In {\em Optimal control, stabilization and nonsmooth analysis},
  volume 301 of {\em Lecture Notes in Control and Inform. Sci.}, pages
  267--282. Springer, Berlin, 2004.

\bibitem{Clarke}
F.~H. Clarke.
\newblock Nonsmooth analysis in systems and control theory.
\newblock In {\em Encyclopedia of Complexity and Systems Science}, pages
  6271--6285. Springer-Verlag, New York, 2009.

\bibitem{ClarketalBook}
F.~H. Clarke, Yu.~S. Ledyaev, R.~J. Stern, and P.~R. Wolenski.
\newblock {\em Nonsmooth analysis and control theory}, volume 178 of {\em
  Graduate Texts in Mathematics}.
\newblock Springer-Verlag, New York, 1998.

\bibitem{FP2003}
F.~Facchinei and J.-S. Pang.
\newblock {\em Finite-Dimensional Variational Inequalities and Complementarity
  Problems}, volume~60 of {\em Springer Series in Operation Research}.
\newblock Springer-Verlag, New York, 2003.

\bibitem{KPS2006}
K.~Kamlibel, J.-S. Pang, and J.~Shen.
\newblock Lyapunov stability of complementarity and extended systems.
\newblock {\em SIAM J. Optim.}, 17(4):1056--1101, 2006.

\bibitem{Kocanetall02}
M.~Kocan and P.~Soravia.
\newblock Lyapunov functions for infinite-dimensional systems.
\newblock {\em J. Funct. Anal.}, 192(2):342--363, 2002.

\bibitem{MordokBook}
B.~S. Mordukhovich.
\newblock {\em Variational analysis and generalized differentiation. {I}},
  volume 330 of {\em Grundlehren der Mathematischen Wissenschaften [Fundamental
  Principles of Mathematical Sciences]}.
\newblock Springer-Verlag, Berlin, 2006.
\newblock Basic theory.

\bibitem{PS08}
J.-S. Pang and D.~E. Stewart.
\newblock Differential variational inequalities.
\newblock {\em Math. Program.}, 113(2, Ser. A):345--424, 2008.

\bibitem{Pazy81}
A.~Pazy.
\newblock The {L}yapunov method for semigroups of nonlinear contractions in
  {B}anach spaces.
\newblock {\em J. Analyse Math.}, 40:239--262 (1982), 1981.

\bibitem{roc70}
R.~T. Rockafellar.
\newblock {\em Convex Analysis}.
\newblock Princeton Mathematical Series, No. 28. Princeton University Press,
  Princeton, N.J., 1970.

\bibitem{RockBook2002}
R.~T. Rockafellar and R.~J.-B. Wets.
\newblock {\em Variational analysis}, volume 317 of {\em Grundlehren der
  Mathematischen Wissenschaften [Fundamental Principles of Mathematical
  Sciences]}.
\newblock Springer-Verlag, Berlin, 1998.

\bibitem{SP94}
D.~Shevitz and B.~Paden.
\newblock Lyapunov stability theory of nonsmooth systems.
\newblock {\em IEEE Trans. Automat. Control}, 39(9):1910--1914, 1994.

\bibitem{simons}
S.~Simons.
\newblock {\em Minimax and monotonicity}, volume 1693 of {\em Lecture Notes in
  Mathematics}.
\newblock Springer-Verlag, Berlin, 1998.

\bibitem{smirnov02}
G.~V. Smirnov.
\newblock {\em Introduction to the theory of differential inclusions},
  volume~41 of {\em Graduate Studies in Mathematics}.
\newblock American Mathematical Society, Providence, RI, 2002.

\end{thebibliography}

\end{document}